\documentclass[10pt]{article}
\usepackage{amsmath,amssymb,amsfonts,amsthm,amsopn,amsbsy,amstext}
\usepackage{mathrsfs}
\usepackage{latexsym}
\usepackage{graphicx}
\usepackage{psfrag}
\usepackage{enumerate}
\usepackage{cases}
\usepackage{multicol}
\usepackage{hyperref}
\usepackage{algorithm}
\usepackage{algpseudocode}
\usepackage{tikz}
\usepackage{color}
\usepackage{lineno}

\usepackage{lineno}%
\def\dbF{{\mathbb{F}}}

\def\dbI{{\mathbb{I}}}

\def\dbN{{\mathbb{N}}}

\def\dbP{{\mathbb{P}}}

\def\dbR{{\mathbb{R}}}
\def\dbS{{\mathbb{S}}}

\def\e{\varepsilon}

\def\3n{\negthinspace \negthinspace \negthinspace }
\def\2n{\negthinspace \negthinspace }
\def\1n{\negthinspace }
\def\ns{\noalign{\smallskip} }

\def\ds{\displaystyle}
\def\cA{{\cal A}}

\def\cC{{\cal C}}
\def\cD{{\cal D}}

\def\cK{{\cal K}}
\def\cL{{\cal L}}

\def\cQ{{\cal Q}}
\def\cR{{\cal R}}
\def\cS{{\cal S}}
\def\cT{{\cal T}}
\def\cU{{\cal U}}

\def\cX{{\cal X}}

\def\mE{{\mathbb{E}}}

\def\ms{\medskip}

\def\q{\quad}
\def\qq{\qquad}

\def\liminf{\mathop{\underline{\rm lim}}}

\def\wt{\widetilde}
\def\cd{\cdot}

\def\({\Big (}
\def\){\Big )}
\def\[{\Big[}
\def\]{\Big]}

\def\={\buildrel \triangle \over =}

\def\ee{\end{equation}}
\def\bea{\begin{eqnarray}}
\def\eea{\end{eqnarray}}
\def\bt{\begin{theorem}}
\def\et{\end{theorem}}
\def\bc{\begin{corollary}}
\def\ec{\end{corollary}}
\def\bl{\begin{lemma}}
\def\el{\end{lemma}}
\def\bp{\begin{proposition}}
\def\ep{\end{proposition}}
\def\br{\begin{remark}}
\def\er{\end{remark}}
\def\ba{\begin{array}}
\def\ea{\end{array}}
\def\bde{\begin{definition}}
\def\ede{\end{definition}}

\newtheorem{lemma}{Lemma}[section]
\newtheorem{remark}{Remark}[section]
\newtheorem{example}{Example}[section]
\newtheorem{theorem}{Theorem}[section]
\newtheorem{corollary}{Corollary}[section]

\newtheorem{definition}{Definition}[section]
\newtheorem{proposition}{Proposition}[section]

\newtheorem{assumption}{Assumption}

\def\punct{}
\newtheoremstyle{dotless}{}{}{\rm}{}{\bf}{\punct}{.5em}{}
\theoremstyle{dotless}

\allowdisplaybreaks[4]
\makeatletter

\@addtoreset{equation}{section}
\makeatother

\title{\bf
Mean-Square Stability of Continuous-Time Stochastic Model Predictive Control
}

\author{
Qi L\"u\thanks{School
of Mathematics, Sichuan University, Chengdu
610064, China. (Email: {\tt lu@scu.edu.cn}). Qi L\"u is supported by the NSF of China under grant 12025105.}
~~~~
Bowen Ma\thanks{School of Mathematical Sciences, Chengdu University of Technology, Chengdu, 610059, China
(Email: {\tt albertmabowen@gmail.com}).}
~~~~
Enrique Zuazua\thanks{Department of Mathematics, Friedrich-Alexander-Universit\"at Erlangen-N\"urnberg (FAU), Chair for Dynamics, Control and Numerics -- Alexander von Humboldt Professorship, 91058 Erlangen, Germany. (Email: {\tt enrique.zuazua@fau.de}).}
\thanks{Departamento de Matem\'aticas, Universidad Aut\'onoma de Madrid, 28049 Madrid, Spain.}
\thanks{Chair of Computational Mathematics, Universidad de Deusto, 48007 Bilbao, Basque Country, Spain. 
Enrique Zuazua was funded by the European Research Council (ERC) under the European Union's Horizon Europe research and innovation programme (grant agreement No.~101096251, project CoDeFeL), the Alexander von Humboldt-Professorship programme, the MODCONFlex Marie Curie Doctoral Network (HORIZON-MSCA-2021-DN-01), the COST Action MAT-DYNNET, the DFG Transregio~154 project, AFOSR grant 24IOE027, the Spanish grants PID2020-112617GB-C22 and TED2021-131390B-I00 (AEI/10.13039/501100011033), and the Madrid Government--UAM Agreement for the Excellence of University Research Staff in the framework of the V~PRICIT.}
}

\date{\today}

\begin{document}
\maketitle
\begin{abstract}

We propose a stochastic model predictive control (SMPC) framework for a broad class of unconstrained controlled stochastic differential equations (SDEs) and establish its mean-square exponential stability in the infinite-horizon limit. At each prediction step of the MPC iteration, the nonlinear controlled SDE is approximated by its linearization at the origin, with the sampled state of the nonlinear system as initial condition, yielding a finite-horizon stochastic linear-quadratic (SLQ) optimal control problem. The resulting optimal control is then applied to the original nonlinear stochastic dynamics until the next sampling instant.

This construction leads to a delayed SMPC scheme whose closed-loop behavior is governed by a coupled time-delay SDE system, a setting that has not been analyzed before. 

We prove global mean-square exponential stability for linear and mildly nonlinear SDEs by exploiting the exponential convergence of the Riccati equation to the algebraic Riccati equation (ARE). For strongly nonlinear SDEs, we establish local mean-square exponential stability by combining exponential Riccati convergence with stopping-time techniques and Gr\"onwall-type estimates.
It is observed that, to ensure the desired local stability properties, the nonlinearities of the SDE are allowed to have polynomial growth but not exponential growth, 
distinguishing SMPC from its deterministic counterpart.

These results provide the first rigorous mean-square stability guarantees for SMPC of SDE systems with delayed state information, thereby advancing the theoretical foundations of stochastic predictive control.
\end{abstract}
\section{Introduction}
{\it Model predictive control} (MPC) is now a well-established and widely adopted feedback control strategy in both industry and academia; 
see \cite{Lee-2011} for a historical survey of its development. 
The origins of MPC can be traced back to the 1950s, when the earliest computer-based supervisory control systems were implemented in various oil and petrochemical industries \cite{Stout-Williams-1995}.
A clear formulation of the underlying principle appeared in Lee and Markus's monograph \cite{Lee-Markus-1967}, who anticipated the MPC framework:
{\it  at each sampling instant, the current system state is measured,  an optimization problem is solved over a finite prediction horizon, 
and only the initial segment of the resulting control input is implemented, after which the procedure is iteratively repeated.}
Despite its insight, this principle initially did not receive much attention due to the high cost and limited computational power at that time. 
It was only with the emergence of affordable and powerful microprocessors and distributed control systems in the mid-1970s that MPC began to gain momentum, 
soon demonstrating remarkable success in industrial practice and attracting huge attention from both industry and academia.

Stability analysis is a central theme in the MPC literature, as one is primarily concerned with whether the closed-loop MPC control strategy stabilizes the system. 
Extensive research has been devoted to this topic in the context of deterministic control systems, including:  
linear or nonlinear, discrete-time or continuous-time, constrained or unconstrained systems.
For comprehensive surveys, tutorials, and extensive references on this foundational topic, we refer the reader to the classical texts \cite{Grune-Pannek-2017,Rawlings-Mayne-Diehl-2019}.

In practice, noise arising from disturbances and model mismatches is inevitable, and hence robustness is indispensable for industrial applications.
Although the feedback nature of MPC provides a certain level of robustness, it does not exploit a priori statistical information about the noise,
which is often available in many problems.

To address these limitations, the {\it stochastic model predictive control} (SMPC) has been developed, with its focus on stochastic control systems.
In this framework, hard constraints on the system variables are reformulated as probabilistic constraints, allowing the controlled system to violate them with prescribed probability.
Research in this area has so far concentrated mainly on stochastic discrete-time systems.
Since our main concern is stochastic continuous-time systems,
we refer interested readers to textbooks \cite{Kouvaritakis-Cannon-2016,Rawlings-Mayne-Diehl-2019} and surveys \cite{Farina-Giilioni-Scattolini-2016,Mayne-2016,Mesbah-2016} for an overview of results for stochastic discrete-time systems.

For stochastic continuous-time systems, the literature is much more limited.
Mahmood and Mhaskar \cite{Mahmood-Mhaskar-2012} proposed a Lyapunov-based nonlinear MPC design for a special class of controlled nonlinear stochastic differential equations,
under the assumption that a stochastic control Lyapunov function exists.
In their approach, the state equation in the MPC optimization problem is a nonlinear ordinary differential equation (ODE) corresponding to the deterministic part of the SDE,
while the Lyapunov function is incorporated into the optimization problem as a constraint to guarantee some local stability property in probability.

For the same class of controlled SDE, Buehler et al. \cite{Buehler-Paulson-Mesbah-2016} proposed an alternative design:
Rather than neglecting the stochastic noise term, they reformulated the constrained stochastic optimization problem in MPC as a constrained optimal control problem for the controlled Fokker-Planck equation.
Wei and Lecchini-Visintini \cite{Wei-Visintini-2014} proved almost sure stability of receding-horizon control with zero control horizon (i.e., instantaneous control) for a class of controlled SDEs with linear growth,
based on an analysis of the associated Hamilton-Jacobi-Bellman (HJB) equation.
There also exists some other MPC design that uses deterministic optimization problems to approximate stochastic ones in MPC \cite{Brok-Madsen-Jørgensen-2018,Volz-Graichen-2015}.

Despite these efforts, research on SMPC for stochastic continuous-time systems is still in its infancy. 
As Mesbah noted in his survey~\cite[p.~40]{Mesbah-2016}, 
``Establishing the theoretical properties of SNMPC algorithms, such as closed-loop stability and constraint satisfaction in the closed-loop sense, poses a great challenge.''

Building on the above considerations, this paper develops a mean-square exponential stability theory for stochastic model predictive control (SMPC) applied to a broad class of unconstrained controlled stochastic differential equations (SDEs).

In contrast to previous works, the MPC strategy is implemented by approximating, at each prediction step, the nonlinear controlled SDE with its linearization at the origin. This yields a finite-horizon stochastic linear-quadratic (SLQ) optimal control problem, whose solution is then applied to the original nonlinear stochastic dynamics.

At every control horizon, the SMPC law is obtained from the SLQ problem using the sampled state of the nonlinear system as the initial condition, and it is implemented in open-loop until the next sampling instant. 
This produces a control law that depends solely on past sampled states, leading to a delayed SMPC scheme whose closed-loop behavior is governed by a coupled time-delay SDE system.

Distinct from the deterministic MPC setting, the notions of sampling and control application in the stochastic case require further clarification (see also Remark \ref{rm2.4'}).
In the theoretical construction of the SMPC algorithm, ��particularly when solving the auxiliary SLQ problems, the sampled state of the physical system is treated as a random variable, and the optimal control produced by the SLQ problem is a stochastic process,
which may create the impression that the implementation is ambiguous.
However, this optimal control is in fact of feedback form and is generated through a deterministic closed-loop optimal strategy. 
In practical implementation,  at each sampling instant,
we sample only a single realized state value, rather than the underlying random variable.
This realized state is then injected into the SDE system associated with the SLQ problem under its closed-loop optimal strategy, producing a realized state trajectory over the sampling interval.
The actual SMPC control applied to the system is obtained by multiplying the deterministic closed-loop optimal strategy with the realized state trajectory generated by the SLQ.

The major advantages of this design are as follows: 
\begin{itemize}
\item  It is substantially simpler than the nonlinear formulation, in which solving the associated Hamilton-Jacobi-Bellman (HJB) equation is required to obtain the optimal control.
In the SLQ setting, by contrast, the optimal control admits a closed-form feedback representation, where the feedback gain matrix is explicitly determined by the solution of the corresponding Riccati equation, for which efficient numerical algorithms are available~\cite[Section~7]{Yong-Zhou-1999}.
\vspace{-2mm}
\item  In general, the optimal controls in stochastic optimal control problems are stochastic processes that are often difficult to implement in practice.
In the SLQ case, however, the control is applied in a deterministic feedback form, with the gain matrix computable offline, which renders the stochastic optimal control problem practically implementable.

Within this framework, the main contribution of the present paper can be stated as follows:
\end{itemize}
\vspace{-5mm}

\begin{enumerate}
\item[(i)]  We establish global mean-square exponential stability of the SMPC algorithm in two settings: linear controlled SDEs and nonlinear controlled SDEs with linear growth under small modeling error (see Theorems \ref{Th2.1}-\ref{Theorem 2.2}).
Inspired by Wonham \cite{Wonham-1968} and Sun et al. \cite{Sun-Wang-Yong-2022}, and \cite{Veldman-Zuazua-2022} in the deterministic frame,  the key step is to prove precise exponential convergence of the solution of the finite-horizon Riccati equation for SLQ with positive semidefinite terminal cost to the positive definite solution of the Algebraic Riccati Equation (ARE) associated with the infinite-horizon SLQ problem.
These exponential convergence estimates, combined with careful analytical arguments, yield the desired mean-square exponential stability. 
\vspace{-2mm}
\item[(ii)] We establish local mean-square exponential stability of the SMPC algorithm for nonlinear controlled SDEs with nonlinear growth (see Theorem \ref{Theorem 2.3}).

The main challenge stems from the nonlinear growth and the delayed feedback structure of SMPC, which together yield a highly coupled nonlinear time-delay SDE system (see \eqref{sta-MPC}-\eqref{Sta-plant}). 
To address this, we first employ stopping-time techniques to confine the SMPC state process within a small ball around the origin. We then derive a Gr\"onwall-type differential inequality with delay and use it to prove mean-square exponential stability of the stopped SMPC state process, provided the ball is small and the gap between the prediction horizon and the effective control horizon is sufficiently large.
This implies that, almost surely, all sample paths of the SMPC state process remain in the small ball. Consequently, for small initial data, the stopped and original SMPC state processes coincide, yielding the desired local mean-square exponential stability. 
Besides, we also reveal an interesting phenomenon distinguishing SMPC from its deterministic counterpart:
In the stochastic setting,
ensuring the local mean-square exponential stability of SMPC permits the nonlinearities of the SDE to exhibit polynomial, but not exponential growth,
whereas no such growth restriction is required for the deterministic case. 
The reason for this phenomenon is that, in SMPC, the control is obtained from the solution of a linear SDE. For a linear SDE, only finite-order moments can be guaranteed to exist, whereas exponential-order moments cannot (see Example \ref{Ex2.1} for details). In contrast, in the deterministic case, the control is derived from the solution of an ODE, so this issue does not arise.
\end{enumerate}

These theoretical findings complement the existing literature by providing, to the best of our knowledge, the first rigorous analysis of this problem. They also validate the computational evidence reported in \cite{Buehler-Paulson-Mesbah-2016} in the setting of the associated Fokker--Planck equation.

The rest of the paper is organized as follows.
In Section 2, we elaborate on the details of the SMPC scheme and summarize the main results of this paper.
Section 3 provides the necessary preliminaries,
where we recall useful results on infinite-horizon SLQ and the associated ARE, 
and establish the exponential convergence of the dynamic Riccati equation to the ARE.
Section 4 contains the detailed proofs of the results presented in Section 2.
Proofs of some technical lemmas are collected in the Appendix at the end of the paper.

We close this section by introducing some frequently used notations.

Let $W(\cdot)$ denote a standard one-dimensional Brownian motion defined on a complete filtered probability space $(\Omega,\mathcal{F},\mathbb{F},\mathbb{P})$, where $\mathbb{F}=\{\mathcal{F}_t\}_{t\geq 0}$ is the augmented natural filtration generated by $W(\cdot)$. 
Mathematical expectations under $\mathbb{P}$ are denoted by $\mathbb{E}$. 
For any $0\leq t_0\leq t_1\leq \infty$, we denote by $L^2_{\mathbb{F}}(t_0,t_1;\mathbb{R}^n)$ the Banach space of $\mathbb{R}^n$-valued, $\mathbb{F}$-adapted stochastic processes $\varphi(\cdot)$ defined on $\Omega\times [t_0,t_1]$ such that
$
\mathbb{E}\left[\int_{t_0}^{t_1}\!|\varphi(s)|^2\,ds\right]<\infty.
$

For any matrix $M\in \mathbb{R}^{n\times m}$, let $|M|$ denote its {\it spectral norm}, i.e., the square root of the largest eigenvalue of $M^\top M$. We write $\mathrm{Tr}(M)$ for the {\it trace} of a square matrix $M\in \mathbb{R}^{n\times n}$, and $\lambda_{\min}(M)$ and $\lambda_{\max}(M)$ for its {\it smallest} and {\it largest eigenvalues}, respectively.

For a positive integer $n$, denote by $\mathbb{S}^n$ the space of all $n\times n$ symmetric matrices, by $\mathbb{S}^n_{>0}$ the subset of positive definite matrices, and by $\mathbb{S}^n_{\geq 0}$ the subset of positive semidefinite matrices.

Finally, for any real number $x\in \mathbb{R}$, let $\lfloor x \rfloor$ denote the {\it greatest integer less than or equal to} $x$.

\vspace{-2mm}


\section{Formulation of the problem and main results}
\vspace{-2mm}

\subsection{Stochastic Model Predictive Control}


Let us first present the formulation of the SMPC problem.
In this framework, we distinguish between the {\it control system} and the {\it plant model}.
The former represents the actual physical dynamics of the system, whose state can only be monitored at discrete time instants $t = k\tau$ 
(for $k \in \dbN$ and {\it sampling interval} $\tau > 0$, which can be adjusted but is generally strictly positive), 
whereas the latter is a continuously observed surrogate model that facilitates the computation of control inputs and enables the practical implementation of the SMPC strategy.

The   control system is the following nonlinear controlled SDE:\vspace{-2mm}
\begin{equation}\label{OriCon}
\begin{cases}
dY(t)=b(Y(t),u(t))dt+\sigma(Y(t),u(t))dW(t),\q t\geq 0,\\
Y(0)=x,
\end{cases}\vspace{-2mm}
\end{equation}
where $b,\sigma:\dbR^n\times \dbR^m\to \dbR^n$ are appropriate functions satisfying $b(0,0)=0,\sigma(0,0)=0$, and $u(\cd)$ is the control.
Our primary goal is to obtain the control that assures the mean-square exponential stability of this system,
even though the state of \eqref{OriCon} can only be observed at discrete time instants.

Instead of constructing a stabilizer directly, we approach the stabilization problem via an auxiliary optimal control formulation.
Inspired by the fact that the optimal control for infinite-horizon SLQ problems also acts as a mean-square exponential stabilizer \cite{Sun-Yong-2019},
we introduce the following auxiliary quadratic performance functional to guide the design of the control law:\vspace{-2mm}
\begin{equation}\label{cost-inf}
J_{\infty}(x ; u(\cdot)) = \frac{1}{2}\, \mathbb{E} \int_0^{\infty} [\langle Q Y(t), Y(t)\rangle + \langle R u(t), u(t)\rangle] \, dt,\vspace{-2mm}
\end{equation}
where $Q \in \mathbb{S}^n_{>0}$ and $R \in \mathbb{S}^m_{>0}$ are suitable weighting matrices.

Here, minimizing \eqref{cost-inf} is not the ultimate goal but rather as a mechanism to generate a feedback law that achieves exponential decay.
Nevertheless, since directly solving this minimization problem for the nonlinear system is generally difficult, 
we consider some surrogate alternatives.
The most natural idea is to linearize \eqref{OriCon} at the origin, solve the algebraic Riccati equation (ARE) for the resulting infinite-horizon SLQ problem, and then test whether the corresponding stabilizing feedback law also stabilizes the original nonlinear system.
This method is indeed effective for many systems under suitably mild assumptions. 
However, this approach requires continuous observation of the true system state, which may not be available in practice.
In numerous practical applications, the system state can only be measured at discrete time instants, not continuously (see \cite[Section 1.4]{Astrom-Wittenmark-1997} for various examples). This constraint may stem from the physical limitations of sensors or be imposed by cost considerations. Consequently, with only sampled state data available, the true state between sampling points remains unknown, rendering any continuous-time state-feedback strategy, including the ARE-based controller, unimplementable in practice.

To overcome this limitation, we introduce the following plant model whose state can be observed continuously, \vspace{-2mm} 
\begin{equation}\label{9.9-eq1}
\ds d X(t)=\big(A X(t)+B u(t)\big) d t+\big(C X(t)+D u(t)\big) d W(t). \vspace{-1mm}
\end{equation}
This controlled linear time-invariant SDE (denoted as $[A, B; C , D]$) is a linearization of the original control system \eqref{OriCon} around the origin,
and can be regarded as an approximate simulator of \eqref{OriCon}.  It serves two roles: 
\begin{itemize}
\item First, the continuously observed state of \eqref{9.9-eq1} provides an approximation of the true state of the control system \eqref{OriCon}, which explains the term stochastic model predictive control (SMPC): the plant model provides a prediction of the control system's behavior in response to a given control input. 
At each sampling time $t=k\tau$, we use the measured state $Y(k\tau)$ of the original system \eqref{OriCon} to initialize the plant model \eqref{9.9-eq1} on the time interval $[k\tau, k\tau + T]$, 
the corresponding solution of \eqref{9.9-eq1} provides a predicted trajectory of the control system over this horizon, where 
$T>0$ is the {\it prediction horizon}.
Since the prediction accuracy decreases as time progresses, the plant model is reinitialized, and a new prediction is computed at the next sampling instant.\vspace{-2mm}
\item Second, we use the continuously observed state of \eqref{9.9-eq1} to generate the control input that drives the control system  \eqref{OriCon} toward the origin. 
{\it Bellman's Principle of Optimality} implies that, if $\bar{u}(\cd)$ is a minimizer of \eqref{cost-inf}, 
then $\bar{u}(\cd)\chi_{[k\tau,\infty)}$
remains a minimizer of\vspace{-2mm} 
\begin{equation}\label{cost-k-inf}
J^{k\tau}_{\infty}(\bar{Y}(k\tau) ; u(\cdot)) = \frac{1}{2}\, \mathbb{E} \int_{k\tau}^{\infty} [\langle Q Y(t), Y(t)\rangle + \langle R u(t), u(t)\rangle] \, dt,
\end{equation}
at each sampling instant $t = k\tau\, (k \in \dbN)$. 
Given that solving these infinite-horizon optimal control problems is generally difficult, we instead consider approximate solutions, which are also expected to retain the desired stability due to the robustness. 
To this end, we introduce the following auxiliary optimal control problems. At each sampling instant $t = k\tau$ ($k \in \dbN$), we replace the original system \eqref{OriCon} with the plant model \eqref{9.9-eq1} and approximate the infinite-horizon cost \eqref{cost-k-inf} by a finite-horizon cost over the interval $[k\tau,k\tau+T]$. 
This leads to a stochastic linear-quadratic (SLQ) optimal control problem. The optimal control of this auxiliary SLQ problem is a continuous state feedback for \eqref{9.9-eq1}, providing an approximate solution to \eqref{cost-k-inf} over $[k\tau,k\tau+T]$. 
This control is then applied to the original system \eqref{OriCon} until the next sampling instant $t = (k + 1)\tau$. This process is repeated iteratively, and the sampling interval $\tau$ is therefore also referred to as {\it control horizon}. Concatenating these interval-wise controls forms the framework of the SMPC design. Therefore, the formulation of \eqref{cost-inf} and the introduction of the plant model \eqref{9.9-eq1} provide the conceptual backbone for this SMPC design.
\end{itemize}

The objective of the {\it stochastic model predictive control (SMPC)} scheme is to use the measurements at time $k\tau$ (i.e., $Y(k\tau)$) together with the plant model and the auxiliary SLQ problem to compute a control function $u(\cdot)$ that stabilizes the origin of~\eqref{OriCon}.
The CENTRAL  problem addressed in this paper is how to design the prediction horizon $T$, the control horizon (or sampling interval) $\tau$, and the auxiliary SLQ problem such that the resulting control $u(\cdot)$ ensures the mean-square exponential stability of system \eqref{OriCon}.

This situation presents a trade-off between the complexity of the control strategy and the cost of measurement. If the state of the system \eqref{OriCon} can be measured continuously over $[0, \infty)$, then the feedback control matrix obtained from the Algebraic Riccati Equation (ARE) can be directly applied to stabilize the system. In contrast, if only discrete-time measurements of the state are available at instances $\{k\tau\}_{k=1}^\infty$, a more sophisticated control strategy, such as Stochastic Model Predictive Control (SMPC), becomes necessary.
Indeed, as discussed at the end of Section 2.2, the ARE-based feedback law can be interpreted as a limiting case of our SMPC strategy. Specifically, by setting $G = P_\infty$ and letting the sampling interval $\tau \to 0$, the SMPC controller converges to the continuous ARE-based feedback law. Therefore, the ARE controller is inherently contained within the SMPC framework, corresponding to the ideal scenario of continuous sampling and full state observation, which enables continuous-time feedback.

One might also consider using the ARE-based feedback control in conjunction with the state of the linearly approximated system \eqref{9.9-eq1} to stabilize the original nonlinear system \eqref{OriCon}. While this approach could succeed for some specific systems, Example \ref{ex-add1} demonstrates that certain systems, which cannot be stabilized by this method, are stabilizable via SMPC.

\vspace{-2mm} 
\begin{remark}
Once this SMPC construction is established, we can guarantee that the admissible control set for \eqref{cost-inf} is non-empty, a fact that is not obvious a priori. 
This could serve as a starting point for a subsequent study of the optimal control problem \eqref{cost-inf} itself, 
including an analysis of the gap between the true minimizer and the iterative approximation produced by this SMPC design. 
Such investigations, however, are beyond the scope of the present paper and are left for future work.
\end{remark}
\vspace{-3mm} 
\subsection{SMPC algorithm}\label{se2.2}

\vspace{-2mm} 

In this subsection, we elaborate on the details of the SMPC algorithm and derive the resulting closed-loop SDE system.

First, we introduce the auxiliary optimal control problem (namely, SLQ problem) that will be used to derive the driving control of the control system \eqref{OriCon}.
Consider the controlled linear SDE,\vspace{-2mm} 
\begin{equation}\label{X_T}
\begin{cases}
\ds dX(t)=\big(A X(t)+B u(t)\big) d t+\big(C X(t)+D u(t)\big) d W(t), \quad t \in [t_1,t_1+T],\\
\ns\ds X(t_1)=x_1,
\end{cases}\vspace{-2mm} 
\end{equation}
and its optimal control problem: 

{\bf Problem (SLQ)$_T$.} Find an optimal control $\bar{u}_T(\cd;t_1,x_1)\in \cU[t_1,t_1+T] =L_{\dbF}^2(t_1,t_1+T;\dbR^m)$ such that \vspace{-2mm}
$$
J_T(t_1,x_1;\bar{u}_T)=\inf_{u\in  \cU[t_1,t_1+T]} J_T(t_1,x_1;u),\vspace{-2mm} 
$$
where \vspace{-3mm} 
\begin{equation}
J_{T}(t_1,x_1; u(\cdot))=\frac{1}{2} \mathbb{E} \[\int_{t_1}^{t_1+T}[\langle Q X(t), X(t)\rangle+\langle R u(t), u(t)\rangle]d t + \langle G X(t_1+T),X(t_1+T) \rangle  \],\vspace{-2mm} 
\end{equation}
is the stage cost, and $G \in \dbS^n_{\geq 0}$ represents the terminal cost. \vspace{-2mm} 
\begin{remark}
The terminal cost $G \in \dbS^n_{\ge 0}$ is introduced to penalize the terminal state, thereby enhancing the stability of the SMPC algorithm.
As will be shown in Subsection 2.3.1, an appropriate choice of $G$ can further strengthen the stability results.
Nevertheless,  the stability results established in this paper hold for any $G\in \dbS^n_{\ge 0}$. 
In particular, the methodology can be applied when $G\equiv 0$ as well, still yielding the desired mean-square exponential decay of the physical system.
In fact,  the terminal cost $G$, the control horizon $\tau$, and the prediction horizon $T$ are three key factors governing the stability of the SMPC algorithm,
and their precise roles will be clarified in subsection \ref{sec-mainre}.  
\end{remark}\vspace{-2mm} 

This Problem (SLQ)$_T$ admits a unique closed-loop optimal control (see \cite[Section 7.1]{Yong-Zhou-1999} for example):\vspace{-2mm} 
\begin{align*}
\bar{u}_T(t;t_1,x_1)=-\left(R+D^{\top} P_T(t-t_1) D\right)^{-1}\left(B^{\top} P_T(t-t_1)+D^{\top} P_T(t-t_1) C\right)\bar{X}_T(t;t_1,x_1),\vspace{-2mm} 
\end{align*}
where $\bar{X}_T(\cd;t_1,x_1)$ is the solution of \eqref{X_T} with the control $\bar{u}_T(\cd;t_1,x_1)$ and $P_T(\cd)\in C^1([0,T];\dbS^n_{>0})$ solves\vspace{-2mm} 
\begin{equation}\label{DRE_T}
\begin{cases}
\dot{P}_T(s)+ \cQ(P_T(s) )  - \cS(P_T(s))^\top  \cR(P_T(s))^{-1} \cS(P_T(s))=0,  \q  s\in [0,T],\\
P_T(T)=G.
\end{cases}\vspace{-2mm} 
\end{equation}
Here we adopt the following conventions, which will be used throughout the paper:
For $P\in \dbS^n$,\vspace{-2mm} 
\begin{equation}\label{nota1}
\begin{aligned}
&\cQ(P):= PA+A^\top P+ C^\top P C+ Q, \qq &\cS(P):=B^\top P+ D^\top P C,\\
&\cR(P):=R+ D^\top P D,  \qq&\cK(P):= -\cR(P)^{-1}\cS(P).
\end{aligned}\vspace{-2mm} 
\end{equation}
\begin{remark} 
Our subsequent stability analysis relies on the asymptotic convergence behavior of $P_T(0)$ as $T$ tends to $\infty$. 
In particular, the limit coincides with the positive definite solution of the corresponding ARE \eqref{ARE} below. 
For convenience, we introduce the following time-reversed Riccati equation: \vspace{-2mm} 
\begin{equation}\label{Riccati}
\begin{cases}
\dot{\Sigma}(s)- \cQ(\Sigma(s) )  + \cS(\Sigma(s))^\top  \cR(\Sigma(s))^{-1} \cS(\Sigma(s))=0, \q s\geq 0,\\
\Sigma(0)=G.
\end{cases}\vspace{-2mm} 
\end{equation}
It is then straightforward to verify that $\Sigma(T-t)=P_T(t)$ for $t\in [0,T]$.
\end{remark}\vspace{-2mm}

Now we present the details of the SMPC algorithm. 
The SMPC control $u_{\rm M}(t)$ is constructed as follows:

(1). We first measure $Y_{\rm M}(0)$. Then we compute the control $\bar{u}_T(t;0,Y_{\rm M}(0))$ on $[0,T]$ and set $u_{\rm M}(t)=\bar{u}_T(t;0,Y_{\rm M}(0))$ on $[0,\tau]$. 
At last, we apply $u_{\rm M}(\cd)$ to the control system.
This leads to the following state trajectory $Y_{\rm M}(t)$ on $[0,\tau]$:\vspace{-2mm} 
\begin{equation*}
\begin{cases}
dY_{\rm M}(t)=b(Y_{\rm M}(t),\Theta(t)\bar{X}_T(t;0,Y_{\rm M}(0)))dt +\sigma(Y_{\rm M}(t),\Theta(t)\bar{X}_T(t;0,Y_{\rm M}(0)))dW(t),\q t\in [0,\tau],\\
Y_{\rm M}(0)=Y_{\rm M}(0),\vspace{-2mm} 
\end{cases}
\end{equation*}
where \vspace{-2mm} 
\begin{equation}\label{Theta}
\Theta(t):=-\left(R+D^{\top} \Sigma(T-t;G) D\right)^{-1}\left(B^{\top}  \Sigma(T-t;G)+D^{\top}  \Sigma(T-t;G) C\right),
\end{equation}
are computed offline via the Riccati equation \eqref{Riccati}.

(2).  We measure $Y_{\rm M}(\tau)$ and compute the control $\bar{u}_T(t;\tau,Y_{\rm M}(\tau))$ on $[\tau,\tau+T]$. Then we 
set $u_{\rm M}(t)=\bar{u}_T(t;\tau,Y_{\rm M}(\tau))$ on $[\tau,2\tau]$ and  apply it to the  control system.
This leads to the following state trajectory $Y_{\rm M}(t)$ on $[\tau,2\tau]$:\vspace{-2mm} 
\begin{equation*}
\begin{cases}
dY_{\rm M}(t)=b(Y_{\rm M}(t),\Theta(t-\tau)\bar{X}_T(t;\tau,Y_{\rm M}(\tau)))dt +\sigma(Y_{\rm M}(t),\Theta(t-\tau)\bar{X}_T(t;\tau,Y_{\rm M}(\tau)))dW(t),\q t\in [\tau,2\tau],\\
Y_{\rm M}(\tau)=Y_{\rm M}(\tau).
\end{cases}\vspace{-2mm} 
\end{equation*}

(3).  Repeating this procedure, we can get the following stochastic dynamics for $Y_{\rm M}(\cd)$:\vspace{-2mm} 
\begin{equation}\label{sta-MPC}
\begin{cases}
dY_{\rm M}(t)\!=\!b\big(Y_{\rm M}(t),\Theta(t\bmod\tau)\bar{X}_T(t;\tau_t,Y_{\rm M}(\tau_t))\big)dt \!+\!\sigma\big(Y_{\rm M}(t),\Theta(t\!\bmod\! \tau)\bar{X}_T(t;\tau_t,Y_{\rm M}(\tau_t))\big)dW(t),& t\geq 0,\\
Y_{\rm M}(\tau_t)=Y_{\rm M}(\tau_t), & t\geq 0,
\end{cases}\vspace{-1mm} 
\end{equation}
where $t \bmod \tau := t- \tau \lfloor \frac{t}{\tau} \rfloor$, 
$\tau_t:=t-(t\bmod \tau)=k\tau$ for $t\in [k\tau, (k+1)\tau)$, $k=1,2,\cdots$, 
and
$\bar{X}_T(t;\tau_t,Y_{\rm M}(\tau_t))$ solves the following SDE:\vspace{-2mm}{\small
\begin{equation}\label{Sta-plant}
\begin{cases}
d\bar{X}_T(s;\tau_t,Y_{\rm M}(\tau_t))=\big[A+ B\Theta(s-\tau_t)\big]\bar{X}_T(s;\tau_t,Y_{\rm M}(\tau_t))ds + \big[C+ D\Theta(s-\tau_t)\big]\bar{X}_T(s;\tau_t,Y_{\rm M}(\tau_t))dW(s),& s\in [\tau_t,\tau_t+T],\\\ns\ds
\bar{X}_T(\tau_t;\tau_t,Y_{\rm M}(\tau_t))=Y_{\rm M}(\tau_t),& t\geq 0.
\end{cases}\vspace{-2mm}
\end{equation}}
Summarizing the above, we get the following SMPC algorithm.

\begin{algorithm}[H]
\caption{Stochastic Model Predictive Control (MPC)} \label{alg:mpc}
\begin{algorithmic}[1]
\Require Prediction horizon $T > 0$, control horizon $0 < \tau \le T$
\State Initialize time index $k \gets 0$
\vspace{0.5em}

\While{control task not finished}
\State \textbf{(Measurement)} Measure the current output/state $ Y_{\mathrm{M}}(k\tau)$. 

\State \textbf{(Prediction and Optimization)} 
Solve  the auxiliary SLQ problem for the plant model with initial data $Y_{\rm M}(k\tau)$ over the time-horizon $[k\tau,\, k\tau+T]$:\vspace{-2mm}
$$
\bar{u}_T(t; k\tau, Y_{\rm M}(k\tau))
= \Theta(t - k\tau)\,
\bar{X}_T(t; k\tau, Y_{\rm M}(k\tau)),
\quad t \in [k\tau, k\tau + T],\vspace{-2mm}
$$
where $\Theta(\cd - k\tau)$ is the closed-loop optimal strategy. 

\State \textbf{(Control Application)}
Apply only the first segment of the computed control:\vspace{-2mm}
$$
u_{\rm M}(t)
= \bar{u}_T(t; k\tau, Y_{\rm M}(k\tau)),
\quad t \in (k\tau, k\tau + \tau].\vspace{-2mm}
$$
Use it to drive the controlled system \eqref{OriCon} on this time interval.

\State \textbf{(Update)} Advance time by $\tau$:\vspace{-2mm}
$$
k \gets k + 1,
$$
and return to Step~2 with the new measurement.
\EndWhile
\end{algorithmic}
\end{algorithm}

Figure 1 shows the workflow of the SMPC algorithm in each time-horizon $t\in [k\tau,(k+1)\tau] \,(k\in \dbN)$.

\begin{figure}[htbp]
\centering
\includegraphics[width=0.40\textwidth]{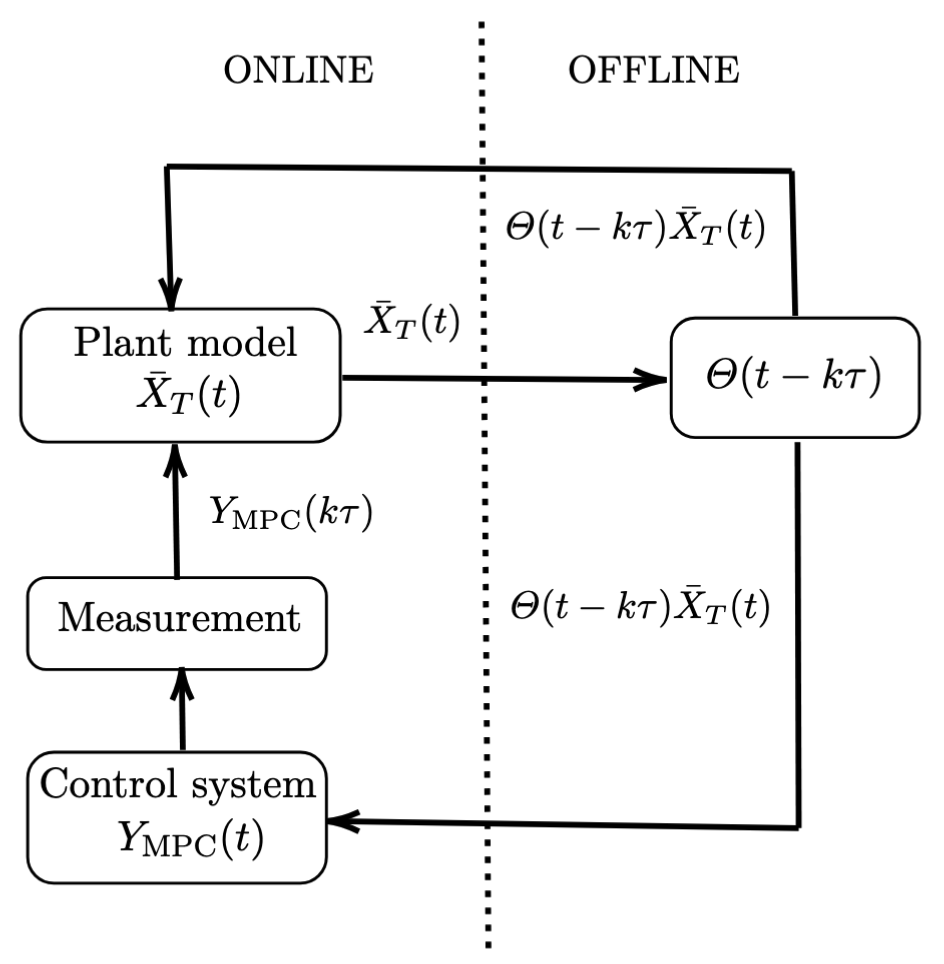} 
\caption{\small The SMPC scheme over time-horizon $t\in [k\tau,(k+1)\tau] \,(k\in \dbN)$ in Algorithm 1. 
The closed-loop optimal strategy $\Theta(t-k\tau)$ for the auxiliary SLQ problem is computed offline. 
The SMPC control $u_{\rm M}(t)=\Theta(t-k\tau)\bar{X}_T(t)$ is applied in feedback form to regulate \eqref{OriCon}, regardless of the realization of the sample path of $\bar{X}_T(t)$.
}
\label{fig:smpc}
\end{figure}

\begin{remark}\label{rm2.4'}
It is necessary to clarify the {\rm\bf Measurement} and {\rm \bf Control Application} steps, which differ substantially from those in deterministic MPC.
Before measurement (i.e., before the observation is made), the state $Y_{\rm M}(k\tau)$ and the applied control $u_{\rm M}(t)$ are random variables:
They admit a distribution but have not yet taken a concrete value. 
This may create the impression that implementation is ambiguous.
In fact, the procedure is fully specified and deterministic once an observation is obtained.

At the sampling instant $k\tau$, we observe the realized value $Y_{\rm M}(k\tau,\omega)$.
We then plug this realized value into the plant model under closed-loop optimal strategy to generate the realized trajectory $\bar{X}_T(t,\omega)$.
The control applied on the interval $[k\tau,(k+1)\tau)$ is obtained deterministically by multiplying the closed-loop optimal strategy $\Theta(t-k\tau)$ on that realized trajectory:
$$
u_{\rm M}(t,\omega)=\Theta(t-k\tau) \bar{X}_{T}(t,\omega),\q t\in [k\tau, (k+1)\tau).
$$
Thus, randomness disappears from the implementation after measurement because the control is computed from observed (realized) data.

Consequently, our SMPC algorithm exhibits both deterministic and stochastic features.
The control strategy, that is, the rule by which the control is computed and applied, is entirely deterministic.
However, the actual control input realized in practice may differ from one implementation to another, because different users of the same SMPC algorithm may encounter different realizations of the driving Brownian motion.
Nevertheless, the results established in Subsection \ref{sec-mainre} can show that, under this strategy,
almost all sample paths of the closed-loop system converge to the origin, since mean-square exponential stability implies almost sure asymptotic stability.

\end{remark}

Finally, we remark that a particular limiting case of the SMPC scheme coincides with the algebraic-Riccati stabilizing feedback law for the corresponding infinite-horizon SLQ problem. 

To this end, we introduce the {\it Algebraic Riccati equation} (see Section~\ref{preli} for details)
\begin{equation}\label{ARE}
\cQ(P_\infty) - \cS(P_\infty)^\top \cR(P_\infty)^{-1}\cS(P_\infty)=0, 
\end{equation}
where $P_\infty$ is its unique positive definite solution and recall that $\Theta_\infty= -\left(R+D^{\top} P_\infty D\right)^{-1}\left(B^{\top}  P_\infty+D^{\top}  P_\infty C\right)$ 
is a stabilizing feedback law for the corresponding infinite-horizon SLQ problem (\eqref{9.9-eq1} together with \eqref{cost-inf}).

By taking
\begin{equation*}
G=P_\infty\text{ and } \tau \to 0,
\end{equation*}
in the SMPC algorithm, the Riccati equation \eqref{Riccati} admits the constant solution $\Sigma(\cd) \equiv P_\infty$ and the associated feedback operator satisfies $\Theta(\cd)\equiv \Theta_\infty.$
Moreover, for $t\in (k\tau,(k+1)\tau]$, we have $\Theta(t-k\tau)\to \Theta(0)=\Theta_\infty$ and $\bar{X}_T(t)\to Y_{\rm M}(t)$,  
so that $u_{\rm M}(t)= \Theta(t-k\tau)\bar{X}_T(t) \to \Theta_\infty Y_{\rm M}(t)$.
Consequently, the limiting dynamics of the SMPC scheme are given by 
\begin{equation}\label{ARE-based-law}
\begin{cases}
dY_{\rm M}(t)=b\big(Y_{\rm M}(t),\Theta_\infty {Y}_{\rm M}(t)\big)dt+\sigma\big(Y_{\rm M}(t),\Theta_\infty {Y}_{\rm M}(t)\big)dW(t),& t\geq 0,\\\ns\ds
Y_{\rm M}(0)=x,
\end{cases}
\end{equation}
which is exactly applying the algebraic-Riccati stabilizing feedback law for the infinite-horizon SLQ.


\subsection{Statements of the main results}\label{sec-mainre}


Before presenting our main results, we make some preparations. First, we introduce the following notations:\vspace{-2mm}
\begin{equation}\label{nota}
\cA_\infty:=A+B\cK(P_\infty), \q \cC_\infty:= C+D\cK(P_\infty), \q \Theta_\infty:= -\left(R+D^{\top} P_\infty D\right)^{-1}\left(B^{\top}  P_\infty+D^{\top}  P_\infty C\right),
\end{equation}
where $P_\infty$ is the unique positive definite solution of ARE \eqref{ARE}.

Next, we recall the definition of  $L^2$-stabilizability of the control system \eqref{9.9-eq1}.
\begin{definition}{\cite{Sun-Yong-2019}}\label{L2-sta}
System \eqref{9.9-eq1} is said to be $L^2$-stabilizable if there exists $\Theta\in \dbR^{m\times n}$ such that the solution to \vspace{-2mm}
\begin{equation*}
\begin{cases}
d X(t)=(A+B\Theta) X(t) dt+(C+D\Theta) X(t)  dW(t), \quad t \geq 0,\\
X(0)=x,
\end{cases}\vspace{-2mm}
\end{equation*}
satisfies \vspace{-2mm}
\begin{equation*}
\mE\int_{0}^{\infty}|X(t;x)|^2dt <\infty.\vspace{-2mm}
\end{equation*}
\end{definition}

We impose the following assumptions:
\begin{assumption}\label{H1}
System \eqref{9.9-eq1} is $L^2$-stabilizable.
\end{assumption}
\begin{remark}\label{rm2.1}
The concept of $L^2$-stabilizability for linear stochastic control systems is introduced in the study of infinite-horizon SLQ to ensure that the set of admissible controls is nonempty--a fundamental requirement for the solvability of the optimal control problem.
This notion also guarantees the existence of at least one deterministic linear feedback control strategy that stabilizes system \eqref{9.9-eq1}.
A necessary and sufficient condition for system \eqref{9.9-eq1} to be $L^2$-stabilizable is that the following ARE admits a positive definite solution \cite[Theorem 3.3.5]{Sun-Yong-2019}:\vspace{-2mm}
$$
PA+A^\top P+ C^\top P C+ I- (PB+C^\top P D) (I +D^\top P D)^{-1} (B^\top P + D^\top P C)=0.\vspace{-2mm}
$$ 
\end{remark}

Finally, we introduce the notations for several important constants  used throughout the paper.

(i) We adopt the convention that $K$ denotes a generic positive constant depending only on system parameters. These constants can be written explicitly, but are denoted by $K$ for simplicity.

(ii) Define\vspace{-2mm}
\begin{equation}\label{nota3}
\begin{aligned}
& K_0:= \frac{n |P_\infty|}{\lambda_{\rm min}(P_\infty) }, \quad
\lambda_\infty:=\frac{\lambda_{\rm min}\big(Q+\cK(P_\infty)^\top R \cK(P_\infty)\big)}{2\lambda_{\rm min}(P_\infty)}, \quad
\lambda^*:=\frac{\lambda_{\min}(P_\infty) \lambda_\infty}{ \lambda_{\max}(P_\infty)}.
\end{aligned}
\end{equation}

(iii) For any  $G\in \dbS^n_{\geq 0}$, we introduce a key bound constant $K_1$, depending on $G$, which quantifies the exponential convergence rate of the solution to \eqref{Riccati} toward $P_\infty$. Specifically, we will show in Lemma \ref{lm2.1} and Lemma \ref{lm2.2} that
$$|\Sigma(t;G)- P_\infty|\leq K_1 e^{-2\lambda_\infty t} \mbox{  for all }t\geq 0.$$
For general, the constant $K_1$ can be defined implicitly as 
\begin{align}\label{K1'}
K_1\!=\! \max\Big \{ \big[|P_\infty|\!+ \!   K_0\big(|G|+      (2\lambda_\infty)^{-1}\big|Q\!+\! \cK(P_\infty)^\top R \cK(P_\infty)\big| \big) \big] e^{4\lambda_\infty t_0} , 
\min\big\{ (4K_0\rho)^{-1},\lambda_\infty (K_0\rho)^{-1}  \big\} e^{2\lambda_\infty t_0}  \Big \},
\end{align}
where \vspace{-4mm}
\begin{align*}
\rho:&= (|B|+|D||C|+|\cK(P_\infty)||D|^2) (|B|+|D||C|){\lambda_{\min}(R)}^{-1} + \big[(|B|+ |D||C|)^2 |D|^2{\lambda_{\min}(R)}^{-2} \\
&\q +   |\cS(P_\infty)|(|B|+ |D||C|) |D|^4 {\lambda_{\min}(R)}^{-3} \big] K_0\big(|G|+   (2\lambda_\infty)^{-1}\big|Q+ \cK(P_\infty)^\top R \cK(P_\infty)\big|\big)
\end{align*}
is a constant depending only on system parameters, and $t_0$ denotes the time at which $\Sigma(t;G)$  enters the ball centered at $P_\infty$ with radius $ (2K_0)^{-1}\min\big\{(4K_0\rho)^{-1}, (K_0\rho)^{-1}\lambda_\infty \big\}$  (In the proof of Lemma \ref{lm2.1} we prove that $\lim_{t\to \infty} \Sigma(t;G)=P_\infty$). The dependence of $K_1$ 
on $G$ is intricate due to the appearance of $t_0$. However, for any fixed $G\in \dbS^n_{\geq 0}$, the value of $t_0$ can be approximated offline via numerical methods for Riccati equations (see \cite[Section 7]{Yong-Zhou-1999}), and thus $K_1$ can, in principle, be computed.

In contrast, if $G \geq P_\infty$, then $K_1$ admits the uniform choice:
\begin{align}
K_1=|G-P_\infty|\frac{n|P_\infty|}{\lambda_{\rm min}(P_\infty)}.
\end{align}

We are now ready to state our main results. The proofs are deferred to Section \ref{proof} for clarity.

\subsubsection{Receding Horizon Control: Perfect Linear Plant Model}

In this subsection, we consider the ideal case where the plant model coincides with the    control system \eqref{OriCon}, that is\vspace{-4mm}
\begin{align*}
b(Y,u)=AY+Bu,\q \sigma(Y,u)=CY+Du.
\end{align*}
This setting is also referred to as {\it Receding Horizon Control} (RHC). 

Note, however, that this limits the dynamics of the physical system to be linear.

Our first result is a stability result for RHC.
\begin{theorem}[Stability of RHC] \label{Th2.1}
Let Assumption \ref{H1} hold. Then there exists a constant $K>0$, independent of $t$, $x$, $T$, $\tau$, and $G$, such that under either of the following conditions:	
\begin{enumerate}
\item[(i)]  $G\geq P_\infty$ and $|G-P_\infty|\leq \min\big\{1,\frac{\lambda_\infty}{2K}\big\}$;
\item[(ii)]   $T-\tau>\frac{1}{2\lambda_\infty}\ln\!\left(\frac{K \sum_{i=1}^{2}K_1^i}{\lambda_\infty}\right)$, 
\end{enumerate}
it holds that\vspace{-2mm}
\begin{equation}
\mE |Y_{\rm M}(t)|^2\leq \frac{\langle P_\infty x,x \rangle}{\lambda_{\rm min}(P_\infty)}e^{ -\lambda_\infty t },\qq \forall t\geq 0.\vspace{-2mm}
\end{equation}
Moreover, in this case, we have\vspace{-2mm}
\begin{equation}
0\leq	J_\infty(x;u_{\rm M}(\cd))-J_\infty(x;u_\infty^*(\cd))\leq K K_1^2 |x|^2 e^{-4\lambda_\infty (T-\tau)},\vspace{-2mm}
\end{equation}
where  $u_\infty^*(t)=\Theta_\infty Y^*_\infty (t)$ and $Y_\infty^*(\cd)$ is the solution of\vspace{-2mm}
\begin{equation}\label{Th2.1-eq2}
\begin{cases}\ds
dY^*_\infty (t)=\cA_\infty Y^*_\infty (t) dt+ \cC_\infty Y^*_\infty (t)dW(t), \q t\geq 0,\\
\ns\ds Y^*_\infty (0)=x,
\end{cases}\vspace{-2mm}
\end{equation}
with $\cA_\infty,\cC_\infty, \Theta_\infty$ given in \eqref{nota}. 
\end{theorem}

\begin{remark}\label{rm2.4}
Under Assumption \ref{H1}, $u^*_\infty(\cd)$ is a minimizer of the cost functional \eqref{cost-inf} for the control system \eqref{OriCon}.
Indeed, $u^*_\infty(\cd)$ and system \eqref{Th2.1-eq2} correspond to the optimal solution of the infinite-horizon SLQ problem introduced in Section \ref{preli}, and it holds that\vspace{-2mm}
\begin{align*}
\inf_{u(\cd)\in \cU_{ad}(x)} J_\infty(x;u(\cd))=  J_\infty(x;u^*_\infty(\cd))=\frac{1}{2}\langle P_\infty x,x \rangle.
\end{align*}
Consequently, in the present case, the cost associated with the SMPC control $u_{\rm M}(\cd)$ is exponentially close to the optimal cost for the corresponding infinite-horizon SLQ problem (i.e., the value function). 

Of course, in the general setting, the situation is more complex, given the intrinsic gap between the nonlinear physical system and the auxiliary linearized one.
\end{remark}

\begin{remark}
Theorem~\ref{Th2.1} shows that, under condition (i), the constraints on the prediction horizon $T$ and the control horizon $\tau$ can be relaxed, whereas under condition (ii), the requirement on the terminal cost $G$ can be weakened. Thus, both the terminal cost $G$ and the horizon gap $(T - \tau)$ play a crucial role in guaranteeing the mean-square stability of the SMPC algorithm. Moreover, if $G = P_\infty$, we immediately obtain $
u_{\rm M}(t) = u_\infty^*(t)$ and 
$Y_{\rm M}(t) = Y_\infty^*(t)$ for all $t \geq 0$. 
\end{remark}

\subsubsection{Stochastic Model Predictive Control:  Linear Growth Case}

In this subsection, we consider the case where the   control system may be nonlinear but satisfies the following linear growth conditions at infinity: For all $(Y,u),(Y',u')\in \dbR^n \times \dbR^m$, it holds that 
\begin{equation}\label{H3}
|b(Y,u)-b(Y',u')-(AY+Bu-AY'-Bu')|\leq L (|Y-Y'|+|u-u'|),
\end{equation}
and 
\begin{equation}\label{H4}
|\sigma(Y,u)-\sigma(Y',u')-(CY+Du-CY'-Du')|\leq L (|Y-Y'|+|u-u'|),
\end{equation}
where $L$ represents the modeling error and, for simplicity, we assume $L \leq 1$.

We then obtain the following result.

\begin{theorem}[Stability of SMPC: Linear growth case]\label{Theorem 2.2}
Let Assumption \ref{H1} hold. Then there exists a constant $K > 0$, independent of $t$, $x$, $T$, $\tau$, $L$, and $G$, such that 
\begin{equation}
\mE|Y_{\rm M}(t)|^2\leq  \frac{\langle P_\infty x,x\rangle}{\lambda_{\rm min}(P_\infty)}e^{\mu_{L,T,\tau}t},
\end{equation}
where  $\mu_{L,T,\tau}:=-2\lambda_\infty+K\big[ L+Le^{K\tau}+L\tau e^{K\tau}+\big(K_1+K_1^2\big)e^{-2\lambda_\infty(T-\tau)} \big]$. 
\end{theorem}
\begin{remark}
When the modeling error $L$ is sufficiently small, one can choose a sufficiently small $\tau$ and a sufficiently large $T$ such that $\mu_{L,T,\tau} < 0$. In this case, the closed-loop SMPC system \eqref{sta-MPC}--\eqref{Sta-plant} achieves global mean-square exponential stability.
Consequently, Theorem \ref{Theorem 2.2} also demonstrates the robustness of the SMPC algorithm for linear control systems.
In particular, if the original   control system is $[A, B; C, D]$ and the corresponding plant model is $[\wt A,\wt B;\wt C,\wt D]$, then as long as	 
$\max \{|A-\wt A|,|B-\wt B|,|C-\wt C|,|D-\wt D|\}$
is small enough, the SMPC algorithm still guarantees global mean-square exponential stability.
\end{remark}

\begin{remark}
Unlike Remark~\ref{rm2.4}, in the present case, the cost associated with the SMPC control $u_{\rm M}(\cd)$ does not need to be exponentially close to the minimum of~\eqref{cost-inf}.
This phenomenon has already been pointed out in \cite{Veldman-Zuazua-2022}.
To illustrate this, consider the control system characterized by $[A, B; C, D]$ and its plant model $[\wt A,\wt B;\wt C,\wt D]$.
Let $\wt P_\infty$ denote the unique positive definite solution of the following ARE
$$
P\wt A+\wt A^\top P+ \wt C^\top P \wt C+ Q- (P \wt B+\wt C^\top P \wt D) (R +\wt D^\top P \wt D)^{-1} (\wt B^\top P + \wt D^\top P \wt C)=0.
$$ 
Then as $T-\tau\to \infty$ and $\tau\to 0$, $Y_{\rm M}(\cd)$ converges to the solution of the following SDE
\begin{equation*}
\begin{cases}\ds
d\wt Y (t)=\big[A- B \big(R+\wt D^{\top} \wt P_\infty \wt D\big)^{-1}\big(\wt B^{\top}  \wt P_\infty+\wt D^{\top}  \wt P_\infty \wt C\big) \big] \wt Y (t) dt \\\ns\ds
\qq \qq + \big[C-D\big(R+\wt D^{\top} \wt P_\infty \wt D\big)^{-1}\big(\wt B^{\top}  \wt P_\infty+\wt D^{\top}  \wt P_\infty \wt C\big) \big] \wt Y (t)dW(t), \q t\geq 0,\\
\ns\ds \wt Y (0)=x,
\end{cases}
\end{equation*}
and $u_{\rm M}(t)$ converges to $\wt u(t)=-\big(R+\wt D^{\top} \wt P_\infty \wt D\big)^{-1}\big(\wt B^{\top}  \wt P_\infty+\wt D^{\top}  \wt P_\infty \wt C\big)\wt Y(t)$.
In contrast, the optimal control and trajectory for  $[A, B; C, D]$ are given by $u^*_\infty(\cd)$  and $Y^*_\infty(\cd)$ \eqref{Th2.1-eq2}.
Clearly, we do not have $\wt Y(\cd)=Y^*_\infty(\cd)$ and $\wt u(\cd)=u^*_\infty(\cd)$.
Therefore, in this case, the SMPC control does not necessarily achieve a cost exponentially close to the optimal one, as the mismatch between the plant model and the control system prevents asymptotic optimality.
\end{remark}

\subsubsection{Stochastic Model Predictive Control: Nonlinear Growth Case}

In this subsection, we address the fully nonlinear case. To this end, we introduce the following assumption:

\begin{assumption}\label{H2}
There exists a positive integer $k\ge2$ such that $b$ and $\sigma$ are $k$-times continuously differentiable, and all their $k$-th order derivatives are bounded.
\end{assumption}

In addition, define \vspace{-2mm}
\begin{equation*}
A=\frac{\partial  b(0,0)}{\partial Y},\q B=\frac{\partial  b(0,0)}{\partial u},\q C=\frac{\partial  \sigma(0,0)}{\partial Y},\q D=\frac{\partial  \sigma(0,0)}{\partial u}.\vspace{-2mm}
\end{equation*}
Under Assumption \ref{H2}, applying Taylor expansion at origin up to order $k$, we can obtain the following estimate for all	  $Y\in \dbR^n, u\in \dbR^m$, \vspace{-3mm}
\begin{equation}\label{H2-esti} 
\big|b(Y,u)-AY-Bu\big| + \big|\sigma(Y,u)-CY-Du\big|\leq K \sum_{i=2}^{k} (|Y|+|u|)^i, \vspace{-2mm}
\end{equation}
where $K$ can be chosen through the maximum of all derivatives of $b$ and $\sigma$ at the origin from the second up to $(k-1)$-th order, together with the uniform bound on their $k$-th order derivatives.
\begin{remark}
Assumption \ref{H2} requires that the nonlinearities in $f$ and $\sigma$ have polynomial growth, which covers a broad class of practical systems.
For instance, if $b$ and $\sigma$ are polynomial functions of the state $Y$ and the control $u$, then they automatically satisfy Assumption \ref{H2}.
Recall that in the SMPC scheme, the applied control $u_{\rm M}(t)=\Theta(t - \tau_t)\bar{X}_T(t; \tau_t, Y_{\rm M}(\tau_t))$  is stochastic and can be unbounded 
even when the state trajectory $Y_{\rm M}(\cd)$ remains within a small ball $B(0,r)$.
Nevertheless, by introducing a stopping time to ensure that the state $Y_{\rm M}(\cd)$ stays inside $B(0,r)$, we obtain
$\mE |u_{\rm M}(t)|^{2i}\leq K \mE | \bar{X}_T(t; \tau_t, Y_{\rm M}(\tau_t))|^{2i}\leq K r^{2i}$.
Thus, although $u_{\rm M}(\cd)$ is unbounded, the expectation of its higher-order nonlinear terms can be made arbitrarily small by choosing $r$ sufficiently small.
Combining this with the boundedness of the state within $B(0,r)$, 
we conclude that the contributions of both the state and control higher-order nonlinearities are negligible and do not affect the stability of the linearized dynamics.
This, in turn, guarantees the stability of the SMPC scheme.
\end{remark}

We now present the following local stability result for SMPC.

\begin{theorem}[Local Stability of SMPC: Nonlinear growth case] \label{Theorem 2.3}
Suppose Assumptions \ref{H1}--\ref{H2} hold. Then there exists a constant $K>0$, independent of $t$, $x$, $\tau$, $T$ and $G$ such that if $T-\tau$ and $x$ satisfy
\begin{equation}\label{Theorem 2.3-eq1}
T-\tau \geq  \frac{1}{2\lambda_\infty}\ln\!\left(\frac{2K(K_1+K_1^2)\lambda_{\max}(P_\infty)}{\lambda_{\min}(P_\infty)\,\lambda_\infty}\right),\q |x|< \min\left\{ \frac{1}{\tau}, \frac{\lambda^*}{4},1, \frac{-K(1+\tau)+\sqrt{K^2(1+\tau)^2+4}}{2} \right\},
\end{equation}  
where $\lambda^*$ and $\lambda_\infty$ are given in \eqref{nota3}, 
then the nonlinear closed-loop SMPC system \eqref{sta-MPC}--\eqref{Sta-plant} admits a unique solution satisfying
\begin{equation}\label{Theorem 2.3-eq2}
\mE |Y_{\rm M}(t )|^2\leq  \frac{2 \langle P_\infty x,x \rangle}{\lambda_{\min}(P_\infty)} e^{-\frac{\lambda^*}{2}t},\q t\geq 0.
\end{equation}
\end{theorem}

\begin{remark}
Assumption \ref{H2} cannot be relaxed to allow exponential growth.
The essential reason is that the control horizon $\tau>0$, which implies the driving control 
$u_{\rm M}(t)=\Theta(t - \tau_t)\bar{X}_T(t; \tau_t, Y_{\rm M}(\tau_t))$ in the SMPC scheme is stochastic and unbounded,
even if the state process $Y_{\rm M}(\cd)$ is stopped inside a small ball $B(0,r)$.
As a consequence, the expectation of the exponential-type nonlinearities in the control may become infinite, in sharp contrast with the polynomial-growth case (i.e., Assumption \ref{H2}), and this destroys the stability of the SMPC scheme.
The following example illustrates this phenomenon.
\begin{example}\label{Ex2.1}
Let $b(Y,u)=e^{-Y}+e^u-2$, $\sigma(Y,u)=e^u-1$ and $x<0$. Then $A=\frac{\partial  b(0,0)}{\partial Y}=-1, B=\frac{\partial  b(0,0)}{\partial u}=1, C=\frac{\partial  \sigma(0,0)}{\partial Y}=0, D=\frac{\partial  \sigma(0,0)}{\partial u}=1.$
By Remark \ref{rm2.1}, the plant model $[A,B;C,D]$ is $L^2$-stabilizable. 
With $Q={5\over 2}, R=1, G=1$, both the Riccati equation \eqref{Riccati} and the ARE \eqref{ARE} admit the solution $\Sigma(\cd)=P_\infty=1$,
yielding the feedback gain $\Theta(\cd)=-{1\over 2}$.
On the first control horizon $[0,\tau]$,  this gives
\begin{equation*}
\begin{cases}
d\bar{X}_T(s;0,x)=-{3\over 2} \bar{X}_T(s;0,x)ds -{1 \over 2}\bar{X}_T(s;0,x) dW(s),\q s\in [0,T],\\
\bar{X}_T(0;0,x)=x,
\end{cases}
\end{equation*}
and therefore $u_{\rm M}(t)=-\frac{x}{2} \exp\big(-\frac{13}{8}t-\frac{1}{2}W(t)\big)$ on $[0,\tau]$. 
Substituting this control into \eqref{OriCon}, we obtain 
\begin{equation*}
\begin{cases}
dY_{\rm M}(t)=\(e^{-Y_{\rm M}(t)}+e^{ -\frac{x}{2} \exp(-\frac{13}{8}t-\frac{1}{2}W(t))} -2\)dt+ \(e ^{-\frac{x}{2} \exp (-\frac{13}{8}t-\frac{1}{2}W(t) )}-1 \)dW(t), \q t\in [0,\tau],\\
Y_{\rm M}(0)=x. 
\end{cases}
\end{equation*}
Since $\mE \left[ e ^{-\frac{x}{2} \exp (-\frac{13}{8}t-\frac{1}{2}W(t) )}\right]=\infty$, the system cannot be mean-square exponentially stable.
\end{example}\noindent
On the other hand, in the following two cases:
\begin{enumerate}
\item when $\sigma=0$, namely, the problem reduces to the deterministic setting;
\item when we are allowed to choose $\tau=0$, so that the time delay vanishes and $u_{\rm M}(t)=\Theta(0) Y_{\rm M}(t)$;
\end{enumerate}
using a stopping time to confine the state process $Y_{\rm M}(\cd)$ automatically restricts the control process  $u_{\rm M}(\cd)$ to a small ball as well.
Consequently, in both situations, Assumption \ref{H2} can be relaxed to requiring only that $b$ and $\sigma$ are twice continuously differentiable, 
which is sufficient to guarantee local mean-square exponential stability.
No boundedness of derivatives or growth conditions are needed.
This also highlights the distinction between MPC in the stochastic setting and its deterministic counterpart.
\end{remark}
\vspace{-2mm}
\begin{remark}\label{rm2.5}
Concerning SMPC for SDEs with nonlinear growth, Mahmood and Mhaskar \cite{Mahmood-Mhaskar-2012} studied the following   control system:\vspace{-2mm}
\begin{equation}\label{rm2.5-eq1}
\begin{cases}
dY(t)=\big[b(Y(t))+g(Y(t))u(t)\big]dt+h(Y(t))dW(t),\q t\geq 0,\\
Y(0)=x,
\end{cases}\vspace{-2mm}
\end{equation}
where the control $u(\cd)$ takes values in a nonempty convex subset $U$ of $\dbR^m$. 
To avoid solving a computationally prohibitive stochastic optimization problem,
they proposed a middle path between nominal MPC and utilizing the additional information of the disturbance.
Specifically, their MPC scheme works as follows: (i) They choose the plant model to be the following constrained ODE with a quadratic cost functional,\vspace{-2mm}
\begin{align}\label{rm2.5-eq2}
\begin{cases}
\begin{cases}
\dot{X}(s)=f(X(s))+g(X(s))u(s),\q s\in [\tau_t,\tau_t+T],\\
X(\tau_t)=Y(\tau_t),\\
\cL V (X(s))\leq -\rho V(X(s)),
\end{cases} \\
J_{T}(\tau_t,Y(\tau_t); u(\cdot))= \int_{\tau_t}^{\tau_t+T}[\langle Q X(s), X(s)\rangle+\langle R u(s), u(s)\rangle]ds,
\end{cases}\vspace{-2mm}
\end{align}
and require the control to be piecewise constant over each control horizon {\rm (}that is, $u(s)=u(k\tau),\,\forall s\in [k\tau,(k+1)\tau)${\rm)}.
Here $\rho>0$, $V:\dbR^n\to \dbR$ is a twice-differentiable stochastic control Lyapunov function {\rm(}assumed to exist{\rm)} 
and $\cL$ denotes the infinitesimal generator of \eqref{rm2.5-eq1}\vspace{-2mm}
\begin{align*}
\cL V(x)=  \nabla V(x)^\top f(x) \;+\; \nabla V(x)^\top g(x) u \;+\; \frac{1}{2}\,\mathrm{Tr} \big( h(x)^\top \nabla^2 V(x)\, h(x) \big).
\end{align*}
Note that the plant model omits the stochastic disturbance term $h(Y(t))dW(t)$, 
while the Hessian term $ \mathrm{Tr} \big( h(x)^\top \nabla^2 V(x)\, h(x) \big)$ in the constraint reintroduces the effect of the stochastic disturbance.
(ii) This deterministic constrained optimization problem \eqref{rm2.5-eq2} is solved on each prediction horizon $[\tau_t,\tau_t+T]$ 
and the minimizing control obtained is implemented over $[\tau_t,\tau_t+\tau]$. 
This procedure is then repeated indefinitely.
Denoting the resulting control and state by $u_{\rm M}(\cd)$ and $Y_{\rm M}(\cd)$, respectively, they proved under certain assumptions the following stability property:

For any $d>0$ and $\lambda\in [0,1)$, there exist $\delta, \beta\in [0,1)$ and $\Delta^*(\lambda)$ {\rm(}depending on $\lambda${\rm)} such that, 
for any control horizon $\tau\leq \Delta^*(\lambda)$ and initial data $x$ belonging to $\mho_\delta:=\{y\in \dbR^n\,|\, \inf_{u\in U}\cL V(y)+\rho V(y)\leq 0  \}\cap \{y\in \dbR^n \,|\, V(y)\leq \delta \} $, 
it holds that \vspace{-3mm}
\begin{equation*}
\dbP\( \sup_{t\geq 0} \langle Q Y_{\rm M}(t),Y_{\rm M}(t) \rangle \leq d \)\geq (1-\beta)(1-\lambda).\vspace{-2mm}
\end{equation*}
Let us compare this with our results:  
(i) Our control system is more general, and the control enters the diffusion term. However, we do not impose any control constraint.
(ii) We do not assume the existence of a stochastic control Lyapunov function when establishing stability properties.
(iii) We establish the local mean-square exponential stability, with explicit stability conditions provided in \eqref{Theorem 2.3-eq1}.  
Although the constants in \eqref{Theorem 2.3-eq1} are not optimal, they are explicit.
(iv) By choosing $G\geq P_\infty$ with $|G-P_\infty|$ sufficiently small, we can make $K_1$ arbitrarily small, 
so that the first inequality in \eqref{Theorem 2.3-eq1} will hold for any $T\geq \tau \geq 0$.
In other words, by choosing $G$ properly, the stability condition becomes independent of $T$ and $\tau$.
\end{remark}

In the SMPC strategy, the nonlinear control system \eqref{OriCon} is first replaced by a linear system \eqref{9.9-eq1}, and then Algorithm \ref{alg:mpc} is introduced to construct a control that stabilizes the state of \eqref{OriCon}. 
A natural question arises: can one bypass this relatively complicated strategy (without repeatedly sampling the true state of \eqref{OriCon}) and simply replace \eqref{OriCon} with \eqref{9.9-eq1}, using the stabilizing control for \eqref{9.9-eq1} to stabilize the original system? The following example of a deterministic control system illustrates why this direct approach generally fails.
\begin{example}\label{ex-add1}
Let $b(Y,u)=Y+\frac{1}{2}Y^2-u$ and $\sigma(Y,u)=0$.
We compare the following two control strategies:\vspace{-2mm}
\begin{enumerate}
\item \textbf{Case 1: Open-loop control constructed directly via the linearized system.}
The system linearized at the origin is\vspace{-2mm}
\begin{equation}\label{Lin-sys}
\begin{cases}
\dot{X}(t)=X(t)-u(t),\quad t\geq 0,\\
X(0)=x.
\end{cases}\vspace{-2mm}
\end{equation}
Choosing $R=\frac{1}{3}$ and $Q=1$  in the infinite-horizon LQ problem yields the stabilizer for \eqref{Lin-sys} as  $u^*(t)=3X(t)=3xe^{-2t}$.
Applying this control to the original system gives\vspace{-2mm}
\begin{equation*}
\begin{cases}
\dot{Y}(t)=Y(t)+\frac{1}{2}Y(t)^2-3xe^{-2t},\quad t\geq 0,\\
Y(0)=x,
\end{cases}\vspace{-2mm}
\end{equation*}
whose explicit solution is\vspace{-2mm}
\begin{equation*}
Y(t) =
\begin{cases}
\sqrt{6x}\, e^{-t} \coth\Big( \sqrt{\frac{3}{2}x} e^{-t} + \text{\rm arccoth}\big( \frac{x+2}{\sqrt{6x}} \big) - \sqrt{\frac{3}{2}x} \Big) - 2, & x>0,\\
0, & x=0,\\
\sqrt{-6x}\, e^{-t} \cot\Big( \sqrt{-\frac{3}{2}x} e^{-t} + \text{\rm arccot}\big( \frac{x+2}{\sqrt{-6x}} \big) - \sqrt{-\frac{3}{2}x} \Big) - 2, & x<0.
\end{cases}\vspace{-2mm}
\end{equation*}
Moreover, it can be shown that\vspace{-2mm}
\begin{equation*}
\begin{cases}
\text{\rm arccoth}\big( \frac{x+2}{\sqrt{6x}} \big) - \sqrt{\frac{3}{2}x}<0, \quad \text{\rm arccoth}\big( \frac{x+2}{\sqrt{6x}} \big)>0, & x>0,\\
\text{\rm arccot}\big( \frac{x+2}{\sqrt{-6x}} \big) - \sqrt{-\frac{3}{2}x}<0, \quad \text{\rm arccot}\big( \frac{x+2}{\sqrt{-6x}} \big)>0, & x<0.
\end{cases}\vspace{-2mm}
\end{equation*}

Because $\text{\rm arccoth}(\cdot)$  and $\text{\rm arccot}(\cdot)$ are singular at the origin, $Y(t)$ blows up in finite time for any $x\neq 0$.

\vspace{-2mm}

\item  \textbf{Case 2: Closed-loop control generated by the SMPC algorithm in the limiting case $\tau \to 0$.} 
Here, $A=1$ and $B=-1$. Choosing $R=\frac{1}{3}$, $Q=1$ and $G=1$, the Riccati equation \eqref{Riccati} admits the solution $\Sigma(t)\equiv 1$.
In the limit $\tau\to 0$, the feedback operator becomes $\Theta(0)=-R^{-1}B^\top \Sigma(T)=3$, leading to the closed-loop control system:\vspace{-2mm}
\begin{equation*}
\begin{cases}
\dot{Y}_{\rm M}(t)=Y_{\rm M}(t)+\frac{1}{2}Y_{\rm M}(t)^2- \Theta(0)Y_{\rm M}(t),\quad t\geq 0,\\
Y_{\rm M}(0)=x,
\end{cases}\vspace{-2mm}
\end{equation*}
with the explicit solution $Y_{\rm M}(t)=\frac{4x}{x-(x-4)e^{2t}}$.
Clearly, $Y_{\rm M}(t)$ is stable for any $x<4$.
Indeed, this closed-loop control is exactly the ARE-based stabilizing feedback law given in \eqref{ARE-based-law}.
\end{enumerate}
\end{example}
Example \ref{ex-add1} demonstrates that directly applying the open-loop stabilizer control derived from the linearized system can be inadequate for achieving stabilization.
This occurs because the discrepancy between the original nonlinear system and its linear approximation may accumulate over time, potentially undermining stability.
In contrast, the SMPC algorithm, through repeatedly state sampling and control adjustment, corrects such mismatches and ensures stability.

Indeed, this example highlights two points:
\begin{enumerate}
\item[(i)] The MPC algorithm is more effective than directly using the open-loop stabilizer of the linearized system for stabilization.
This highlights the necessity of applying MPC, particularly because multiple-time sampling allows the control to continuously correct the mismatch between the linearized and the original system, thereby ensuring stability.

\item[(ii)]Closed-loop control outperforms open-loop control in the long term due to its feedback nature.
In the open-loop case, the control is a stabilizer of the linearized system. As time progresses, the mismatch between the linearized system and the original system may grow large, causing this control to fail in stabilizing the original system.
In contrast, closed-loop control uses the real-time sampled state to correct the control continuously, mitigating the mismatch and maintaining stability.
\end{enumerate}

\subsection{Main ideas for the proofs of Theorem  \ref{Th2.1}--\ref{Theorem 2.3}}
In this subsection, we outline the proof strategies for Theorems \ref{Th2.1}--\ref{Theorem 2.3} and summarize the main steps.
The analysis relies on the exponential convergence of the Riccati equation \eqref{Riccati} toward the algebraic Riccati equation (ARE) \eqref{ARE} (see Lemma \ref{lm2.1}--\ref{lm2.2}),
as well as the mean-square exponential stability of the SDE  $dX(t)=\cA_\infty X(t)dt+ \cC_\infty X(t) dW(t)$ (see Lemma \ref{lm3.2}).

\begin{enumerate}
\item [(i)]{\bf Proof of Theorem \ref{Th2.1}:} Step 1: Since the plant model \eqref{9.9-eq1} coincides with the   control system \eqref{OriCon}, 
we can get the stochastic dynamics of SMPC:\vspace{-2mm}
\begin{align*}
\begin{cases}
dY_{\rm M}(t)=\cA_{T,\tau}(t)Y_{\rm M}(t)dt+ \cC_{T,\tau}(t)Y_{\rm M}(t)dW(t),\q t\geq 0,\\
Y_{\rm M}(0)=x,
\end{cases}\vspace{-2mm}
\end{align*}
where $\cA_{T,\tau}(t) = A + B \Theta(t \bmod \tau)$ and $\cC_{T,\tau}(t) = C + D \Theta(t \bmod \tau)$.

Step 2: The exponential convergence of Riccati equation implies $| \Theta(t \bmod \tau)-\Theta_\infty| \leq  K K_1 e^{-2\lambda_\infty (T-\tau)}$.
Then $\cA_{T,\tau}(t)\to \cA_\infty$ and $\cC_{T,\tau}(t)\to \cC_\infty$ as long as $K_1\to 0$ or $T-\tau\to \infty$,
from which we can obtain the desired mean-square exponential stability.
\item[(ii)]{\bf Proof of Theorem  \ref{Theorem 2.2}:} Step 1: Denote $\e(t)= \bar{X}_T(t;\tau_t,Y_{\rm M}(\tau_t))-Y_{\rm M}(t)$ to be the prediction error. 
Then the analysis of the SMPC dynamics of $Y_{\rm M}(t)$ \eqref{sta-MPC} and $\bar{X}_T(t;\tau_t,Y_{\rm M}(\tau_t))$ \eqref{Sta-plant} can be transformed to 
the coupled system of $Y_{\rm M}(t)$ and $\e(t)$.

Step 2: Using Lipschitz conditions \eqref{H3}--\eqref{H4} and Gr\"onwall's inequality, we have \vspace{-2mm}
\begin{equation}\label{se2.3-eq1}
\mE|\e(t)|^2\leq KL^2 e^{K\tau} \int_{\tau_t}^{t} \mE|Y_{\rm M}(s)|^2ds,
\end{equation}
which implies that the effect of prediction error is small if the modeling error $L$ is small.
Step 3: Applying It\^o's formula to $s\mapsto \langle P_\infty Y_{\rm M}(s),Y_{\rm M}(s) \rangle$, and utilizing \eqref{se2.3-eq1}, Lipschitz conditions \eqref{H3}--\eqref{H4} and Gr\"onwall's inequality,
we can establish the result.
\item[(iii)] {\bf Proof of Theorem  \ref{Theorem 2.3}:} 
Step 1: Similar to Step 1 in (ii), we analyze the coupled system of $Y_{\rm M}(t)$ and $\e(t)$.
Fix $r>0$, and let the initial data $x$ belong  to the ball $B(0,r):=\{y\,|\, |y|<r \}$.  
We define the stopping time $\zeta :=\inf \{t\geq 0\,|\, |Y_{\rm M}(t)|\geq r \}$ and introduce the stopped processes $y(t):= Y_{\rm M}(t \wedge\zeta )$ and $ \psi (t):=\e(t \wedge \zeta )$.
Then it follows that \vspace{-2mm}
\begin{equation}\label{se2.3-eq2}
|y(t)|\leq r,\q t\geq 0.
\end{equation}
Step 2: Define $\Psi(t) := \mE \langle P_\infty y(t), y(t) \rangle$ for $t\geq 0$.
Fix any $0\leq t_0\leq t_1<\infty$ and apply It\^o's formula to $s\mapsto \langle P_\infty y(t),y(t)  \rangle$ over $[t_0,t_1]$. 
Using \eqref{se2.3-eq2} and \eqref{H2-esti}, after lengthy but straightforward computations, 
we can get that $\Psi(t)$ satisfies a Gr\"onwall-type inequality involving time delay.
When $r$ is small and $T-\tau$ is large,
we show that $\Psi(t)$ enjoys the mean-square exponential stability, which leads to 
\begin{equation}\label{se2.3-eq3}
\mE |Y_{\rm M}(t\wedge \zeta)|^2 \leq \frac{2 \langle P_\infty x,x \rangle}{\lambda_{\min}(P_\infty)} e^{-\frac{\lambda^*}{2}t},\q t\geq 0.
\end{equation}
Step 3: By Fatou's lemma, property \eqref{se2.3-eq3} ensures that, almost surely, all sample paths of $Y_{\rm M}(\cd)$ remain within the small ball $B(0,r)$.
Consequently, for small initial data, the stopped process $y(\cd)$ coincides with the original SMPC state process  $Y_{\rm M}(\cd)$, 
yielding the desired local mean-square exponential stability.

\end{enumerate}
\vspace{-4mm}

\section{Preliminary results}\label{preli}

\subsection{Infinite horizon SLQ  and the related ARE}

In this subsection, we collect useful results on the infinite-horizon  SLQ  and the associated  ARE  \eqref{ARE}.

To begin with, consider the control system \eqref{9.9-eq1} with initial condition $X(0)=x\in \dbR^n$ and the cost functional \eqref{cost-inf}. The control $u(\cd)$  is required to belong to the {\it admissible control set} \vspace{-2mm}
\begin{equation*}
\cU_{ad}(x)=\left\{ u\in L_{\dbF}^2([0,\infty);\dbR^m) \,\Big|\, \mathbb{E} \int_0^{\infty}|X(t ; x, u)|^2 d t<\infty \right\},\vspace{-2mm}
\end{equation*}
such that the cost functional is well defined.
Under Assumption \ref{H1}, the set is nonempty \cite[Theorem 3.3.5]{Sun-Yong-2019}.

We now introduce the following optimal control problem:

{\bf Problem (SLQ)$_{\infty}$:} For each $x \in \mathbb{R}^n$, find $\bar{u}(\cdot) \in \cU_{ad}(x)$ such that
$$
J_{\infty}(x ; \bar{u}(\cdot))=\inf _{u(\cdot) \in \cU_{ad}(x)} J_{\infty}(x ; u(\cdot)) := V_{\infty}(x).
$$

The key to solving Problem (SLQ)$_\infty$ lies in the ARE \eqref{ARE}. We recall the following known results.

\begin{lemma}{\cite[Proposition 3.1]{Sun-Wang-Yong-2022}}\label{lm3.1}
Let Assumption \ref{H1} hold. Then the ARE \eqref{ARE} admits a unique solution   $P_\infty$ in $\dbS_{>0}^n$.
Moreover, for any $x\in \dbR^n$, Problem (SLQ)$_{\infty}$ admits a unique optimal control $\bar{u}(\cd)$
given by
\begin{equation*}
\bar{u}(t)=-(R+D^\top P_\infty D)^{-1}(B^\top P_\infty+ D^\top P_\infty C) \bar{X}(t),
\end{equation*}
and the value function takes the following quadratic form:
\begin{equation*}
V_{\infty}(x)=\frac{1}{2}\langle P_\infty x,x \rangle,\q \forall x\in \dbR^n.
\end{equation*}
\end{lemma}

Next, we establish a useful estimate.
Using the notations in \eqref{nota1} and \eqref{nota}, the ARE \eqref{ARE} can be rewritten as:\vspace{-2mm}
\begin{align}\label{ARE-form2}
Q + P_\infty \cA_\infty + \cA_\infty^\top P_\infty + \cC_\infty^\top P_\infty \cC_\infty + \cK(P_\infty)^\top R \cK(P_\infty) = 0,\vspace{-2mm}
\end{align}
and the following estimate holds.

\begin{lemma}\label{lm3.2}
Let \vspace{-3mm}
\begin{align}\label{lm2.2-Phi}
\begin{cases}
d\Phi(t)=\cA_\infty \Phi(t)dt+ \cC_\infty \Phi(t) dW(t),\q t\geq 0,\\
\Phi(0)=I_n,
\end{cases}\vspace{-2mm}
\end{align}
then \vspace{-2mm}
\begin{equation}
\mE |\Phi(t)|^2\leq K_0 e^{-2\lambda_\infty t},\q \forall t \geq 0,\vspace{-2mm}
\end{equation}
where $\lambda_\infty$ and $K_0$ are given in \eqref{nota3}.
\end{lemma}

\begin{proof}
Applying It\^o's formula to $s \mapsto \Phi(s)^\top P_\infty \Phi(s)$, we obtain for all $t>0$:\vspace{-2mm}
\begin{align*}
\mE \Phi(t)^\top P_\infty \Phi(t)- P_\infty&=\mE \int_{0}^{t} \Phi(s)^\top \(\cA_\infty^\top P_\infty +P_\infty \cA_\infty +\cC_\infty^\top P_\infty \cC_\infty\)\Phi(s)ds\\
&=-\mE \int_{0}^{t} \Phi(s)^\top \(Q+ \cK(P_\infty)^\top R \cK(P_\infty) \)\Phi(s)ds.\vspace{-2mm}
\end{align*}
Consequently,  \vspace{-2mm}
\begin{align*}
\mE |\Phi(t)|^2&=\mE \lambda_{\rm max}\(\Phi(t)^\top \Phi(t) \)\leq \mE {\rm Tr}\( \Phi(t)^\top \Phi(t)\) 
\\&\leq \frac{1}{\lambda_{\rm min}(P_\infty) } \mE {\rm Tr}\( \Phi(t)^\top P_\infty \Phi(t)\)\\
&=  \frac{1}{\lambda_{\rm min}(P_\infty) } {\rm Tr} \(\mE \Phi(t)^\top P_\infty \Phi(t)\)\\
&\leq   \frac{1}{\lambda_{\rm min}(P_\infty) } \[ {\rm Tr}(P_\infty)- \mE \int_{0}^{t} {\rm Tr}\(\Phi(s)^\top \big(Q+ \cK(P_\infty)^\top R \cK(P_\infty) \big)\Phi(s)\)ds\]\\
&\leq   \frac{1}{\lambda_{\rm min}(P_\infty) } \[ {\rm Tr}(P_\infty)-  \lambda_{\rm min} \big(Q+ \cK(P_\infty)^\top R \cK(P_\infty) \big) \mE \int_{0}^{t} {\rm Tr} \(\Phi(s)^\top \Phi(s)\)ds\]\\
&\leq \frac{1}{\lambda_{\rm min}(P_\infty) } \[ n|P_\infty| -  \lambda_{\rm min} \big(Q+ \cK(P_\infty)^\top R \cK(P_\infty) \big) \mE \int_{0}^{t}  \big|\Phi(s)\big|^2ds\]\\
&\leq \frac{n |P_\infty|}{\lambda_{\rm min}(P_\infty) }e^{ -2\lambda_\infty t}.\vspace{-2mm}
\end{align*}
\vspace{-2mm}
\end{proof}

\begin{lemma}{\cite[Proposition 4.2]{Sun-Wang-Yong-2022}}\label{MPC-lm1} 
Let Assumption  \ref{H1} hold,  $P_\infty$ be the unique positive definite solution of ARE \eqref{ARE}
and $\Sigma(\cd;0)$ be the solution of Riccati equation \eqref{Riccati} with $G=0$. Then \vspace{-2mm}
\begin{equation}
\lim_{t\to \infty} \Sigma(t;0)=P_\infty.   
\end{equation}
\end{lemma}
\vspace{-3mm}

\subsection{Exponential convergence of solutions of the Riccati equation \eqref{Riccati} toward solutions of the ARE \eqref{ARE}}

In this subsection, we establish several useful estimates for the Riccati equation \eqref{Riccati} and prove the exponential convergence of its solution to the positive definite solution of the ARE \eqref{ARE}.

\begin{lemma}\label{lm5.1}
Let Assumption \ref{H1} hold. Then the solution $\Sigma(\cdot;G)$ of \eqref{Riccati} satisfies the following estimate:\vspace{-2mm}
\begin{equation}
\sup_{t \geq 0} | \Sigma(t;G) | \leq K_0 \left( |G| + \frac{ \big| Q + \cK(P_\infty)^\top R \cK(P_\infty) \big| }{2\lambda_\infty} \right),\vspace{-2mm}
\end{equation}	
where $K_0$ and $\lambda_\infty$ are given in \eqref{nota3}.
\end{lemma}

\begin{lemma}\label{lm2.1}
Let Assumption \ref{H1} hold, let $P_\infty$ be the unique positive definite solution of the ARE \eqref{ARE}, and let $\Sigma(\cdot;G)$ be the solution of the Riccati equation \eqref{Riccati}. Then for any $G \in \dbS^n_{\geq 0}$, there exists a constant $K_1 > 0$, depending on $G$ and system parameters but independent of $t$, such that\vspace{-2mm}
\begin{equation}\label{Con-Riccati}
|\Sigma(t;G) - P_\infty| \leq K_1 e^{-2\lambda_\infty t}, \quad \forall t \geq 0,\vspace{-2mm}
\end{equation}
where $\lambda_\infty$ is given in \eqref{nota3}.
\end{lemma} 

\begin{remark}\label{rm2.2}
Two key constants appear in Lemma \ref{lm2.1}: $K_1$ and $\lambda_\infty$. From \eqref{nota3}, it follows that the convergence rate $\lambda_\infty$ is independent of $G$, since the ARE \eqref{ARE} itself does not depend on $G$. In contrast, the bound constant $K_1$ depends delicately on the choice of $G$ and is given explicitly in \eqref{K1'} (see also \eqref{K1} in the proof of Lemma \ref{lm2.1} for its derivation). Therefore, for any $G \in \dbS^n_{\geq 0}$, the solution of \eqref{Riccati} converges exponentially to $P_\infty$ with a uniform rate $-2\lambda_\infty$, although the bounding constant $K_1$ varies with $G$.
\end{remark}

Lemma \ref{lm2.1} is previously established for the special case $G = 0$ in \cite[Theorem 4.1]{Sun-Wang-Yong-2022}. Here, we extend this result to arbitrary initial data $G \in \dbS^n_{\geq 0}$, noting that $G$--which represents the terminal cost--contributes positively to the stability of the SMPC. Moreover, by a more careful selection of $G$, we can derive an explicit and more tractable bound for $K_1$. 

\begin{lemma}\label{lm2.2}
Let Assumption \ref{H1} hold, let $P_\infty$ be the unique positive definite solution of the ARE \eqref{ARE}, and let $\Sigma(\cdot;G)$ be the solution of the Riccati equation \eqref{Riccati}. Then for any $G \geq P_\infty$, it holds that\vspace{-2mm}
\begin{equation*}
|\Sigma(t;G) - P_\infty| \leq K_1(G) e^{-2\lambda_\infty t}, \quad t \geq 0,
\end{equation*}
where\vspace{-2mm}
\begin{equation}\label{lm3.4-K1}
K_1(G) = |G - P_\infty| \cdot \frac{n |P_\infty|}{\lambda_{\rm min}(P_\infty)}.
\end{equation}
\end{lemma}
Proofs of these three lemmas are provided in the Appendix (Subsections \ref{app-sec1}--\ref{app-sec3}). 

\subsection{A Gr\"onwall-type  inequality involving time delay}

In this subsection,  we establish the following result, which can be regarded as a Gr\"onwall-type inequality involving time delay. 
\begin{lemma}\label{lm-tau}
Suppose $k>0$,  $0<r\leq \min\{\frac{1}{\tau},\frac{k}{4}\}$ and the following differential inequality holds\vspace{-2mm}
\begin{equation}
y(t)'\leq -(k-r)y(t)+ry(\tau_t),\q t\geq 0.\vspace{-2mm}
\end{equation}	
Then  we can get \vspace{-2mm}
\begin{equation}
y(t)\leq 2y(0)e^{-\frac{k}{2}t}.\vspace{-2mm}
\end{equation}

\end{lemma}
We believe Lemma \ref{lm-tau} should be a known result. However, since we were unable to find an exact reference, we provide a proof for completeness.
\begin{proof}[Proof of Lemma \ref{lm-tau}]
For $t\in [0,\tau]$, we have
\begin{align*}
y(t)&\leq y(0)+\int_{0}^{t} \big[ -(k-r)y(s)+ry(0)\big]ds \leq (1+r\tau)y(0)+ \int_{0}^{t} \big[ -(k-r)y(s)\big]ds \leq  (1+r\tau)y(0)e^{-(k-r)t},
\end{align*}
which implies  \vspace{-1mm}
$$
y(\tau)\leq (1+r\tau)y(0)e^{-(k-r)\tau}.\vspace{-1mm}
$$
For $t\in [\tau, 2\tau]$, we obtain\vspace{-1mm}
\begin{align*}
y(t)&\leq y(\tau)+\int_{\tau}^{t} \( -(k-r)y(s)+ry(\tau)\)ds \leq  (1+r\tau)y(\tau) +\int_{\tau}^{t} \( -(k-r)y(s)\)ds\\
&\leq (1+r\tau)y(\tau)e^{-(k-r)(t-\tau)} \leq  (1+r\tau)^2y(0)e^{-(k-r)t},\vspace{-1mm}
\end{align*}
which yields\vspace{-1mm}
$$
y(2\tau)\leq (1+r\tau)^2y(0)e^{-(k-r)2\tau}.\vspace{-2mm}
$$
Inductively,  we obtain for $t\in [n\tau,(n+1)\tau]$ that\vspace{-1mm}
$$
y(t)\leq (1+r\tau)^{n+1}y(0)e^{-(k-r)t},\vspace{-1mm}
$$
which is equivalent to \vspace{-1mm}
$$
y(t)\leq (1+r\tau)^{\lfloor \frac{t}{\tau} \rfloor+1}y(0)e^{-(k-r)t},\q t\geq 0.\vspace{-1mm}
$$
Since $r\leq \min\{\frac{1}{\tau},\frac{k}{4}\}$, it follows that for all $t\geq 0$,\vspace{-1mm}
\begin{align*}
y(t) \leq 2y(0)e^{\lfloor \frac{t}{\tau} \rfloor \ln (1+r\tau) -(k-r) t} \leq 2y(0) e^{ \frac{t}{\tau}  r\tau -(k-r) t} \leq 2y(0)e^{-\frac{k}{2}t}.
\end{align*}
\end{proof}

\section{Conclusions, perspectives and open problems}
For a broad class of unconstrained stochastic control systems, 
we propose a SMPC framework that reformulates the finite-horizon optimization problem as an SLQ problem,
by linearizing the nonlinear SDE system at the origin.
Under explicit conditions, we establish the global mean-square exponential stability for both linear SDEs and nonlinear SDEs with linear growth (under small modeling error),
and obtain the local mean-square exponential stability for nonlinear SDEs with nonlinear growth.
The analysis is mainly based on the exponential convergence of the Riccati equation \eqref{Riccati} toward the ARE \eqref{ARE},
and the obtained results clearly show the influence of the key parameters in the SMPC algorithm:
the prediction horizon $T$, the control horizon $\tau$ and the terminal cost $G$.
Indeed, $T-\tau$ and $G$ are two factors contributing to the stability of SMPC,
and choosing smaller $\tau$ can reduce the influence of the modeling error $L$.

Several points deserve further discussion:\vspace{-2mm}

\begin{enumerate}
\item {\bf Discrete-time system}. This paper focuses on establishing a theory of mean-square exponential stability for continuous-time systems. It is anticipated that analogous results for the discrete-time case can be derived similarly. The crucial step involves deriving a discrete-time counterpart to Lemma \ref{lm2.1}. While a deterministic version of this lemma is available in \cite[Appendix C]{Veldman-Zuazua-2022}, as far as we know, its stochastic analogue remains absent from the literature.
\vspace{-2mm}
\item {\bf Other notions of stochastic stability}. Various notions of stability exist for stochastic systems \cite{Khaminskii-2012, Mao-2007}, including asymptotic stability in probability, moment exponential stability, and almost sure exponential stability. Developing a stability theory for SMPC across these different definitions presents a promising research direction. A particularly intriguing case is almost sure exponential stability, wherein stochastic noise can exhibit a stabilizing effect \cite[Section 4.5]{Mao-2007}; that is, a system that is unstable in its deterministic form may be stabilized by the introduction of stochastic perturbations. We conjecture that this phenomenon likely extends to SMPC, further underscoring the theoretical importance of its study.
\vspace{-2mm}
\item {\bf Turnpike property}. The results of this paper are based on the exponential convergence of the Riccati equation \eqref{Riccati} to its associated algebraic Riccati equation (ARE) \eqref{ARE}. A key advantage of this approach lies in the explicitness of both the bounding constant and the convergence rate, which provides an effective framework for deriving precise stability criteria. Furthermore, this exponential Riccati convergence is also instrumental in establishing stochastic turnpike properties \cite{Sun-Wang-Yong-2022}, revealing a profound connection between SMPC and the turnpike theory for stochastic systems. Given that exponential Riccati convergence implies the stochastic turnpike property--indicating that the latter is a weaker notion--it is compelling to investigate how the turnpike property, in turn, influences the stability of continuous-time SMPC. Initial explorations in this direction for deterministic systems are documented in \cite{Grune-2013, Pan-Stomberg etal-2021}.\vspace{-2mm}
\item {\bf  Control problem with control/state constraints}. A primary limitation of this work is the assumption of an unconstrained control and state in the control system. To address control constraints, a natural extension would be to integrate findings from constrained stochastic linear-quadratic (SLQ) theory and analyze their asymptotic behavior. For example, in the one-dimensional case with the control restricted to a closed cone, Hu and Zhou \cite{Hu-Zhou-2005} derived two extended Riccati equations to construct the optimal control and characterize the value function. By establishing the exponential convergence of these extended equations, the results presented herein could likely be generalized to such control-constrained, one-dimensional systems. Regarding state constraints, which are commonly recast as probabilistic or expectation-type constraints in stochastic frameworks, guaranteeing both stability and constraint satisfaction remains a significant challenge.
\vspace{-2mm}
\item {\bf Nonlinear plant model}.  Another perspective for addressing nonlinear SDE systems is to formulate the finite-horizon optimization problem as a nonlinear stochastic optimal control problem,
that is, to adopt a plant model identical to the nonlinear control system.
Following the same line of thought as in this paper, establishing the corresponding stability properties requires the use of Hamilton-Jacobi theory and an analysis of its asymptotic convergence behavior.
An attempt in this direction can be found in \cite{Wei-Visintini-2014}.

\vspace{-2mm}
\item {\bf Learning-based SMPC}. The findings of this paper presuppose an accurate model of the control system, underscoring that precise modeling is paramount for the success of the proposed SMPC algorithm. In practical applications, however, model uncertainty is common, stemming from factors such as insufficient data, limited model classes, or the absence of an explicit mathematical description of the true dynamics. This naturally leads to the question of how to implement control when an accurate model is unavailable. The rapid advancement of computational power and machine learning has spurred the development of {\it learning-based MPC/SMPC} to address this challenge. To the best of our knowledge, existing literature in this domain focuses exclusively on discrete-time stochastic systems (see the recent survey \cite{Hewing-Wabersich-Menner-Zeiliger-2020} and the references therein). Notably, no published results currently exist for learning-based SMPC applied to continuous-time stochastic control systems.

\end{enumerate}

\section{Appendix}

\subsection{Proofs of the main results}\label{proof}

This subsection is devoted to proofs of Theorem \ref{Th2.1}--\ref{Theorem 2.3}.

\subsubsection{Stability of RHC (Theorem \ref{Th2.1})}

\begin{proof}[Proof of Theorem \ref{Th2.1}]	
Comparing the dynamics of $Y_{\rm M}(t)$ in \eqref{sta-MPC} and $\bar{X}_T(t;0,Y_{\rm M}(0))$ in \eqref{Sta-plant} over $t \in [0,\tau]$, we find that $Y_{\rm M}(t) = \bar{X}_T(t;0,Y_{\rm M}(0))$ for $t \in [0,\tau]$. By induction, we   get that
\begin{align*}
Y_{\rm M}(t)=\bar{X}_T(t;k\tau,Y_{\rm M}(k\tau)),\q t\in [k\tau,(k+1)\tau],\q k=1,2\cdots,
\end{align*}
Substituting  this into \eqref{sta-MPC}, we obtain 
\begin{align}\label{pr-th2.1-eq2}
\begin{cases}
dY_{\rm M}(t)=\cA_{T,\tau}(t)Y_{\rm M}(t)dt+ \cC_{T,\tau}(t)Y_{\rm M}(t)dW(t),\q t\geq 0,\\
Y_{\rm M}(0)=x,
\end{cases}
\end{align}
where $\cA_{T,\tau}(t) = A + B \Theta(t \bmod \tau)$ and $\cC_{T,\tau}(t) = C + D \Theta(t \bmod \tau)$ are $\tau$-periodic matrix functions, and $\Theta(\cdot)$ is given in \eqref{Theta}.

To establish the stability of $Y_{\rm M}(\cd)$, we rewrite \eqref{pr-th2.1-eq2} as
\begin{align*}
dY_{\rm M}(t)&=\big[\cA_\infty Y_{\rm M}(t)+ \big( \cA_{T,\tau}(t)-\cA_\infty \big)Y_{\rm M}(t) \big]dt + \big[\cC_\infty Y_{\rm M}(t)+ \big( \cC_{T,\tau}(t)-\cC_\infty \big)Y_{\rm M}(t) \big]dW(t).
\end{align*}
Using \eqref{ARE-form2} and applying It\^o's formula to $t \mapsto \langle P_\infty Y_{\rm M}(t), Y_{\rm M}(t) \rangle$, we obtain for all $t > 0$ that 
\begin{align*}
&\mE \left[ \langle P_\infty Y_{\rm M}(t),Y_{\rm M}(t) \rangle - \langle P_\infty x,x \rangle\right]\\
&= \mE \int_{0}^{t} \big\{  -\langle (Q+\cK(P_\infty)^\top R \cK(P_\infty) )Y_{\rm M}(s),   Y_{\rm M}(s)  \rangle\\
& \qq+ \big\langle \big[ 2 P_\infty  (\cA_{T,\tau} (s)\!-\!\cA_\infty)\!+\!2 \cC_\infty ^\top\! P_\infty (\cC_{T,\tau} (s)\!-\!\cC_\infty ) \!+\! (\cC_{T,\tau} (s)\!-\!\cC_\infty )^\top P_\infty (\cC_{T,\tau} (s)\!-\!\cC_\infty ) \big]Y_{\rm M}(s), Y_{\rm M}(s) \big \rangle \big \}ds.
\end{align*}
From Lemma \ref{lm5.1}, we have 
\begin{equation}\label{Theta-Thetainf}
\begin{aligned}
&|\Theta(t\bmod \tau) -\Theta_\infty|\\ &\leq \big(|B|+|D||C| \big)K_1 e^{-2\lambda_\infty (T-\tau)}  \Big(\frac{|D|^2}{\lambda_{\min}(R)^2}|\Sigma(T-t\bmod \tau)|+\frac{1}{\lambda_{\min}(R)} \Big)\\
&\leq  \big(|B|+|D||C| \big)K_1 e^{-2\lambda_\infty (T-\tau)}  \Big(\frac{|D|^2}{\lambda_{\min}(R)^2}\big(K_0+ (2\lambda_\infty)^{-1}|Q+ \cK(P_\infty)^\top R \cK(P_\infty)| \big)+\frac{1}{\lambda_{\min}(R)} \Big)\\
&\leq K K_1 e^{-2\lambda_\infty (T-\tau)},
\end{aligned}
\end{equation}
where $K$ is a generic positive constant independent of $t, x, \tau, T, G$, depending only on system parameters, which can be written out explicitly.

For any $s\in [0,\infty)$, we have
\begin{align*}
& \frac{1}{\lambda_{\min}(P_\infty)} \big|2 P_\infty  (\cA_{T,\tau} (s)-\cA_\infty)+2 \cC_\infty ^\top P_\infty (\cC_{T,\tau} (s)-\cC_\infty )+ (\cC_{T,\tau} (s)-\cC_\infty )^\top P_\infty (\cC_{T,\tau} (s)-\cC_\infty )  \big|\\
&\leq \frac{1}{\lambda_{\min}(P_\infty)}  \big[ 2 \big( |P_\infty||B| + |\cC^\top_\infty P_\infty||D| \big)|\Theta(t \bmod \tau)-\Theta_\infty|+|P_\infty||D|^2 |\Theta(t \bmod \tau) -\Theta_\infty|^2   \big]\\
&\leq  KK_1e^{-2\lambda_\infty (T-\tau)}+ K K_1^2 e^{-4\lambda_\infty (T-\tau)}\\
&\leq  K e^{-2\lambda_\infty (T-\tau)}\big( K_1+K_1^2\big).
\end{align*}
Let $\dbI:=K e^{-2\lambda_\infty (T-\tau)}\big( K_1+K_1^2\big)$. By Gr\"onwall's inequality, we   obtain that
\begin{align*}
\mE |Y_{\rm M}(t)|^2\leq \frac{\langle P_\infty x,x \rangle}{\lambda_{\rm min}(P_\infty)}e^{(-2\lambda_\infty+\dbI )t}.
\end{align*}
Moreover, if $G \geq P_\infty$, then by \eqref{lm3.4-K1} in Lemma \ref{lm2.2}, $K_1 \leq K |G - P_\infty|$. Hence, 
$$
\dbI \leq K e^{-2\lambda_\infty (T-\tau)} \sum_{i=1}^{2}|G-P_\infty|^i.
$$
Therefore, under either of the following conditions:
\begin{enumerate}
\item[(i)]    $G\geq P_\infty$ with $|G-P_\infty|\leq \min\big\{1,\frac{\lambda_\infty}{2K}\big\}$, 
\item[(ii)]   $T-\tau$ satisfying $T-\tau > \frac{1}{2\lambda_\infty}\ln\!\left(\frac{K \sum_{i=1}^{2}K_1^i}{\lambda_\infty}\right),$
\end{enumerate}
we have\vspace{-2mm}
\begin{equation}\label{pr-th2.1-eq1}
\dbI \leq \lambda_\infty\q \text{and }\q \mE |Y_{\rm M}(t)|^2\leq \frac{\langle P_\infty x,x \rangle}{\lambda_{\rm min}(P_\infty)}e^{ -\lambda_\infty t }.\vspace{-2mm}
\end{equation}

Next, we prove the convergence result. From \eqref{pr-th2.1-eq1}, $u_{\rm M}(\cdot)$ is an admissible control. Applying It\^o's formula to $t \mapsto \frac{1}{2} \langle P_\infty Y_{\rm M}(t), Y_{\rm M}(t) \rangle$ over $[0, T]$, we get that\vspace{-2mm}
\begin{align}
&\frac{1}{2} \mE \langle P_\infty Y_{\rm M}(T), Y_{\rm M}(T) \rangle - \frac{1}{2} \mE \langle P_\infty x, x \rangle \nonumber\\
&= \frac{1}{2} \mE \int_0^T \left[ \langle (P_\infty A + A^\top P_\infty + C^\top P_\infty C) Y_{\rm M}(s), Y_{\rm M}(s) \rangle \right. \label{MPC-eq1}\\
&\quad + 2 \langle (B^\top P_\infty + D^\top P_\infty C) Y_{\rm M}(s), u_{\rm M}(s) \rangle + \langle D^\top P_\infty D u_{\rm M}(s), u_{\rm M}(s) \rangle \bigg] ds.\nonumber\vspace{-2mm}
\end{align}
Letting $T \to \infty$ in \eqref{MPC-eq1} and using Lemma \ref{lm3.1}, we obtain\vspace{-2mm}
\begin{align*}
0 &\leq J_\infty(x; u_{\rm M}(\cdot)) - J_\infty(x; u_\infty^*(\cdot)) = J_\infty(x; u_{\rm M}(\cdot)) - \frac{1}{2} \langle P_\infty x, x \rangle \\
&= \frac{1}{2} \mE \int_0^\infty \left[ \langle (P_\infty A + A^\top P_\infty + C^\top P_\infty C + Q) Y_{\rm M}(s), Y_{\rm M}(s) \rangle \right. \\
&\quad + 2 \langle (B^\top P_\infty + D^\top P_\infty C) Y_{\rm M}(s), u_{\rm M}(s) \rangle + \langle (R + D^\top P_\infty D) u_{\rm M}(s), u_{\rm M}(s) \rangle \bigg] ds \\
&= \frac{1}{2} \mE \int_0^\infty \left[ \langle \Theta_\infty^\top (R + D^\top P_\infty D) \Theta_\infty Y_{\rm M}(s), Y_{\rm M}(s) \rangle \right. \\
&\quad - 2 \langle (R + D^\top P_\infty D) \Theta_\infty Y_{\rm M}(s), u_{\rm M}(s) \rangle + \langle (R + D^\top P_\infty D) u_{\rm M}(s), u_{\rm M}(s) \rangle \bigg] ds \\
&= \frac{1}{2} \mE \int_0^\infty \langle (R + D^\top P_\infty D) (u_{\rm M}(s) - \Theta_\infty Y_{\rm M}(s)), u_{\rm M}(s) - \Theta_\infty Y_{\rm M}(s) \rangle ds \\
&\leq K \int_0^\infty |\Theta(s \bmod \tau) - \Theta_\infty|^2 \mE |Y_{\rm M}(s)|^2 ds \\
&\leq K K_1^2 e^{-4\lambda_\infty (T - \tau)} \frac{\langle P_\infty x, x \rangle}{\lambda_{\min}(P_\infty)} \int_0^\infty e^{-\lambda_\infty s} ds \\
&\leq K |x|^2 K_1^2 e^{-4\lambda_\infty (T - \tau)}.\vspace{-2mm}
\end{align*}
This completes the proof.
\end{proof}

\subsubsection{Stability of SMPC: Linear Growth Case (Theorem \ref{Theorem 2.2})}

\begin{proof}[Proof of Theorem \ref{Theorem 2.2}]
Note that the equation for $Y_{\rm M}(\cd)$ \eqref{sta-MPC} can be rewritten as:\vspace{-2mm}
\begin{align*}
dY_{\rm M}(t)=&  \big[\cA_\infty Y_{\rm M} (t) \! + \! ( \cA_{T,\tau}(t)-\cA_\infty ) Y_{\rm M}(t)\!+\! B \Theta(t\bmod \tau)\e (t) +  b\big(Y_{\rm M}(t),\Theta(t\bmod\tau)\bar{X}_T(t;\tau_t,Y_{\rm M}(\tau_t))\big)\\
&- A Y_{\rm M}(t)-B \Theta(t\bmod \tau) \bar{X}_T(t;\tau_t,Y_{\rm M}(\tau_t)) \big ]dt\\
&+ \big[ \cC_\infty Y_{\rm M} (t)\!+\!  ( \cC_{T,\tau}(t)\!-\!\cC_\infty ) Y_{\rm M}(t)\!+\! D \Theta(t\bmod \tau)\e (t)\! +\!\sigma\big(Y_{\rm M}(t),\Theta(t\bmod \tau)\bar{X}_T(t;\tau_t,Y_{\rm M}(\tau_t))\big) \\ 
&\q-C Y_{\rm M}(t) -D \Theta(t\bmod \tau) \bar{X}_T(t;\tau_t,Y_{\rm M}(\tau_t))  \big ]dW(t),\vspace{-2mm}
\end{align*}
where $\e(t)= \bar{X}_T(t;\tau_t,Y_{\rm M}(\tau_t))-Y_{\rm M}(t)$ satisfies\vspace{-2mm}
\begin{align*}
\begin{cases}
d\e(t)= \big[\cA_\infty \e (t)+  ( \cA_{T,\tau}(t)-\cA_\infty ) \e(t)- B \Theta(t\bmod \tau)\e (t) -  b\big(Y_{\rm M}(t),\Theta(t\bmod\tau)\bar{X}_T(t;\tau_t,Y_{\rm M}(\tau_t))\big) \\
\hspace{4em}+ A Y_{\rm M}(t)+B \Theta(t\bmod \tau) \bar{X}_T(t;\tau_t,Y_{\rm M}(\tau_t)) \big ]dt\\\ns\ds 
\hspace{4em}+ \big[ \cC_\infty \e (t)+  ( \cC_{T,\tau}(t)-\cC_\infty ) \e(t)- D \Theta(t\bmod \tau)\e (t) -\sigma\big(Y_{\rm M}(t),\Theta(t\bmod \tau)\bar{X}_T(t;\tau_t,Y_{\rm M}(\tau_t))\big) \\
\hspace{4em}\q +C Y_{\rm M}(t)+D \Theta(t\bmod \tau) \bar{X}_T(t;\tau_t,Y_{\rm M}(\tau_t))  \big ]dW(t),\q  t\geq 0,\\
\e(\tau_t)=0,
\end{cases}\vspace{-2mm}
\end{align*}
Applying It\^o's formula to $s\mapsto \langle P_\infty \e(s),\e(s) \rangle$, we obtain for all $t>0$ that\vspace{-2mm}
\begin{align*}
&\mE \langle P_\infty \e(t),\e(t) \rangle \\&= \mE \int_{\tau_t}^{t} \Big[  - \big \langle  (Q+\cK(P_\infty)^\top R \cK(P_\infty))\e(s),\e(s) \big\rangle + 2\big\langle  P_\infty    \big[  ( \cA_{T,\tau}(s)-\cA_\infty ) \e(s)- B \Theta(s\bmod \tau)\e (s)\\
& \qq\q - \! b\big(Y_{\rm M}(s),\Theta(s\bmod\tau)\bar{X}_T(s;\tau_t,Y_{\rm M}(\tau_s))\big)\! +\! A Y_{\rm M}(s) \!+\!B \Theta(s\bmod \tau) \bar{X}_T(s;\tau_s,Y_{\rm M}(\tau_s))\big],\e(s)  \big \rangle \\
&\q+ 2\big\langle  \cC_\infty^\top P_\infty  \big[  ( \cC_{T,\tau}(s)-\cC_\infty ) \e(s)- D \Theta(s\bmod \tau)\e (s)\\
& \qq\q -\!  \sigma\big(Y_{\rm M}(s),\Theta(s\bmod\tau)\bar{X}_T(s;\tau_t,Y_{\rm M}(\tau_s))\big)\! +\! C Y_{\rm M}(s) \!+\! D \Theta(s\bmod \tau) \bar{X}_T(s;\tau_s,Y_{\rm M}(\tau_s))\big],\e(s)  \big \rangle \\
&\q + \big\langle   P_\infty  \big[  ( \cC_{T,\tau}(s)\!-\!\cC_\infty ) \e(s)\!-\! D \Theta(s\bmod \tau)\e (s) \!-\! \sigma\big(Y_{\rm M}(s),\Theta(s\bmod\tau)\bar{X}_T(s;\tau_t,Y_{\rm M}(\tau_s))\big) \!+\! C Y_{\rm M}(s)\\
&\qq\q +D \Theta(s\bmod \tau) \bar{X}_T(s;\tau_s,Y_{\rm M}(\tau_s))\big],\big[  ( \cC_{T,\tau}(s)-\cC_\infty ) \e(s) - D \Theta(s\bmod \tau)\e (s) \\
&\qq\q  -  \sigma\big(Y_{\rm M}(s),\Theta(s\bmod\tau)\bar{X}_T(s;\tau_t,Y_{\rm M}(\tau_s))\big)  + C Y_{\rm M}(s) +D \Theta(s\bmod \tau) \bar{X}_T(s;\tau_s,Y_{\rm M}(\tau_s))\big]     \big\rangle \Big]ds.
\end{align*}
By \eqref{Theta-Thetainf} and \eqref{H3}--\eqref{H4}, we have $|\cA_{T,\tau}(s) - \cA_\infty| + |\cC_{T,\tau}(s) - \cC_\infty| \leq K K_1 e^{-2\lambda_\infty (T - \tau)}$ and \vspace{-2mm}
\begin{align}
&\big|b\big(Y_{\rm M}(s),\Theta(s\bmod\tau)\bar{X}_T(s;\tau_t,Y_{\rm M}(\tau_s))\big) - A Y_{\rm M}(s) -B \Theta(s\bmod \tau) \bar{X}_T(s;\tau_s,Y_{\rm M}(\tau_s))\big| \nonumber\\
&+ 	\big|\sigma\big(Y_{\rm M}(s),\Theta(s\bmod\tau)\bar{X}_T(s;\tau_t,Y_{\rm M}(\tau_s))\big) - C Y_{\rm M}(s) -D \Theta(s\bmod \tau) \bar{X}_T(s;\tau_s,Y_{\rm M}(\tau_s))\big|\nonumber \\
&\leq KL\big( |Y_{\rm M}(s)|+ |\e(s)| \big),\nonumber
\end{align}
where $K$ is a generic constant that can be written out explicitly and does not depend on $t,x,T,\tau,L,G$. 
Using these estimates, we  obtain that\vspace{-2mm}
\begin{align*}
\mE|\e(t)|^2&\leq    \int_{\tau_t}^{t} \big\{ -2\lambda_\infty \mE|\e(s)|^2+ K\mE |\e(s)|^2+ KL \,\mE \big[\big(|Y_{\rm M}(s)|+|\e(s)|\big)|\e(s)|\big] \\
&\qq \q +  KL^2 \mE \big[ |Y_{\rm M}(s)|^2 + |\e(s)|^2 \big]\big\}ds\\
& \leq \int_{\tau_t}^{t} \big( -2\lambda_\infty \mE|\e(s)|^2+ K(1+L)\mE |\e(s)|^2+ KL^2 \mE |Y_{\rm M}(s)|^2 \big)ds,
\end{align*}
where we  used $L\leq 1$.
From Gr\"onwall's inequality,  we have
\begin{align}\label{pr-th2.2-eq1}
\mE|\e(t)|^2\leq KL^2e^{K(1+L)\tau} \int_{\tau_t}^{t} \mE|Y_{\rm M}(s)|^2ds\leq KL^2 e^{K\tau} \int_{\tau_t}^{t} \mE|Y_{\rm M}(s)|^2ds.
\end{align}
Similarly, applying It\^o's formula to $s\mapsto \langle P_\infty Y_{\rm M}(s),Y_{\rm M}(s) \rangle$, 
we obtain for all $t>0$ that 
\begin{align*}
\mE|Y_{\rm M}(t)|^2&\leq \frac{\langle P_\infty x,x\rangle}{\lambda_{\rm min}(P_\infty)}- \int_{0}^{t}  2\lambda_\infty \mE|Y_{\rm M}(s)|^2ds \\
&\q+ \int_{0}^{t} K\big( |\cA_{T,\tau}(s)-\cA_\infty| + |\cC_{T,\tau}(s)-\cC_\infty|+ |\cC_{T,\tau}(s)-\cC_\infty|^2 +L+L^2\big) \mE|Y_{\rm M}(s)|^2ds\\
&\q +\int_{0}^{t} (K+KL)\mE \big[|Y_{\rm M}(s)|^2 \big]^{1/2} \mE \big[|\e(s)|^2 \big]^{1/2}ds+\int_{0}^{t} (K+KL^2)\mE |\e(s)|^2 ds\\
&\leq \frac{\langle P_\infty x,x\rangle}{\lambda_{\rm min}(P_\infty)}- \int_{0}^{t}  2\lambda_\infty \mE|Y_{\rm M}(s)|^2ds \\
&\q+ \int_{0}^{t} K\big( |\cA_{T,\tau}(s)-\cA_\infty| + |\cC_{T,\tau}(s)-\cC_\infty|+ |\cC_{T,\tau}(s)-\cC_\infty|^2 +L\big) \mE|Y_{\rm M}(s)|^2ds\\
&\q +\int_{0}^{t} K\mE \[|Y_{\rm M}(s)|^2 \]^{1/2} \mE \[|\e(s)|^2 \]^{1/2}ds+\int_{0}^{t} K\mE |\e(s)|^2 ds.
\end{align*}
Substituting \eqref{pr-th2.2-eq1} and using \eqref{Theta-Thetainf}, we get
\begin{align}\label{pr-th2.2-eq2}
\mE|Y_{\rm M}(t)|^2&\leq \frac{\langle P_\infty x,x\rangle}{\lambda_{\rm min}(P_\infty)}+ \int_{0}^{t} -2\lambda_\infty \mE|Y_{\rm M}(s)|^2ds \nonumber \\
&\q+ \int_{0}^{t} K\big( |\cA_{T,\tau}(s)-\cA_\infty| + |\cC_{T,\tau}(s)-\cC_\infty|+ |\cC_{T,\tau}(s)-\cC_\infty|^2+L\big) \mE|Y_{\rm M}(s)|^2ds\nonumber\\
&\q +\int_{0}^{t} KL e^{K\tau}\( \mE |Y_{\rm M}(s)|^2 +  \int_{\tau_s}^{s} \mE|Y_{\rm M}(r)|^2dr \)ds+\int_{0}^{t} KLe^{K\tau}   \int_{\tau_s}^{s} \mE|Y_{\rm M}(r)|^2dr  ds,\nonumber\\
&\leq \frac{\langle P_\infty x,x\rangle}{\lambda_{\rm min}(P_\infty)}+ \int_{0}^{t} -2\lambda_\infty \mE|Y_{\rm M}(s)|^2ds +\int_{0}^{t} KL e^{K\tau}  \int_{\tau_s}^{s} \mE|Y_{\rm M}(r)|^2dr  ds \nonumber\\
&\q + \int_{0}^{t} K\big( L+Le^{K\tau}+ K_1e^{-2\lambda_\infty(T-\tau)} +K_1^2e^{-4\lambda_\infty(T-\tau)} \big) \mE|Y_{\rm M}(s)|^2ds. 
\end{align}
For any $t \in [k\tau, k\tau + \tau]$, we have 
\begin{equation*}
\begin{aligned}
\int_{0}^{t}    \int_{\tau_s}^{s} \mE|Y_{\rm M}(r)|^2dr  ds&= 	\int_{k \tau}^{t}    \int_{k\tau}^{s} \mE|Y_{\rm M}(r)|^2drds 
+\sum_{l=0}^{k-1} \int_{l \tau}^{(l+1)\tau}    \int_{l\tau}^{s} \mE|Y_{\rm M}(r)|^2drds\\
&\leq  \tau   \int_{k\tau}^{t} \mE|Y_{\rm M}(r)|^2dr+ \tau\sum_{l=0}^{k-1} \int_{l \tau}^{(l+1)\tau}    \mE|Y_{\rm M}(r)|^2dr \\
&\leq \tau \int_{0}^{t} \mE|Y_{\rm M}(r)|^2dr.
\end{aligned}
\end{equation*}
Combining this with \eqref{pr-th2.2-eq2} yields\vspace{-2mm}
\begin{align*}
\mE|Y_{\rm M}(t)|^2\!\leq \! \frac{\langle P_\infty x,x\rangle}{\lambda_{\rm min}(P_\infty)}\!-\! \int_{0}^{t}\!  2\lambda_\infty \mE|Y_{\rm M}(s)|^2ds\! +\! \int_{0}^{t}\! K\big[ L\!+\!Le^{K\tau}\!\!+\!L\tau e^{K \tau} \!+\!\big(K_1\!+\!K_1^2\big)e^{-2\lambda_\infty(T\!-\!\tau)}\big] \mE|Y_{\rm M}(s)|^2ds.
\end{align*}
Thus, we obtain\vspace{-3mm}
\begin{align*}
\mE|Y_{\rm M}(t)|^2&\leq  \frac{\langle P_\infty x,x\rangle}{\lambda_{\rm min}(P_\infty)}e^{\mu_{L,T,\tau}t},
\end{align*}
where $\mu_{L,T,\tau}=-2\lambda_\infty+K\big[ L+Le^{K\tau}+L\tau e^{K\tau}+\big(K_1+K_1^2\big)e^{-2\lambda_\infty(T-\tau)} \big] $.
\end{proof}

\subsubsection{Local Stability of SMPC: Nonlinear Growth Case (Theorem \ref{Theorem 2.3})}

First, recall from \eqref{sta-MPC}--\eqref{Sta-plant} that the closed-loop SMPC system is given by
\begin{equation}\label{state-MPC}
\begin{cases}
d Y_{\rm M}(t)=    b\big(Y_{\rm M}(t),\Theta(t\bmod\tau)\bar{X}_T(t;\tau_t,Y_{\rm M}(\tau_t))\big)dt\\\ns\ds 
\hspace*{5em}+\sigma\big(Y_{\rm M}(t),\Theta(t\bmod \tau)\bar{X}_T(t;\tau_t,Y_{\rm M}(\tau_t))\big) dW(t), & t\geq 0,\\
d\bar{X}_T(s;\tau_t,Y_{\rm M}(\tau_t))=\big[A+ B\Theta(s-\tau_t)\big]\bar{X}_T(s;\tau_t,Y_{\rm M}(\tau_t))ds \\\ns\ds
\hspace*{10em}+ \big[C+ D\Theta(s-\tau_t)\big]\bar{X}_T(s;\tau_t,Y_{\rm M}(\tau_t))dW(s),& s\in [\tau_t,\tau_t+\tau],\\\ns\ds
Y_{\rm M}(0)=x,\q \bar{X}_T(\tau_t;\tau_t,Y_{\rm M}(\tau_t))=Y_{\rm M}(\tau_t),& t\geq 0.
\end{cases}
\end{equation}

Since $f(\cd,\cd)$ and $\sigma(\cd,\cd)$ are twice continuously differentiable, they are locally Lipschitz.   By \cite[Theorem 16.7.6]{Cohen-Elliott-2015}, system \eqref{state-MPC} admits a unique local solution.

Fix $r>0$, and let the initial data $x$ belong  to the ball $B(0,r):=\{y\,|\, |y|<r \}$.  
Define the stopping time
\begin{equation}\label{stopping}
\zeta :=\inf \{t\geq 0\,|\, |Y_{\rm M}(t)|\geq r \},
\end{equation}
and introduce the stopped processes
$$\e (t):= \bar{X}_T(t;\tau_{t},Y_{\rm M}(\tau_{t}))-Y_{\rm M}(t),\q y(t):= Y_{\rm M}(t \wedge\zeta ),\q \psi (t):=\e(t \wedge \zeta ).
$$
With these definitions, we rewrite system \eqref{state-MPC} as 
{\small
\begin{align*}
\begin{cases}
d Y_{\rm M}(t)=  \Big[\cA_\infty Y_{\rm M} (t) \! + \! ( \cA_{T,\tau}(t)-\cA_\infty ) Y_{\rm M}(t)\!+\! B \Theta(t\bmod \tau)\e (t)\\
\hspace{6em}+  b\big(Y_{\rm M}(t),\Theta(t\bmod\tau)\bar{X}_T(t;\tau_t,Y_{\rm M}(\tau_t))\big)\! -\! A Y_{\rm M}(t)-B \Theta(t\bmod \tau) \bar{X}_T(t;\tau_t,Y_{\rm M}(\tau_t)) \Big ]dt\\\ns\ds  
\hspace{6em}+ \Big[ \cC_\infty Y_{\rm M} (t)+  ( \cC_{T,\tau}(t)-\cC_\infty ) Y_{\rm M}(t)+ D \Theta(t\bmod \tau)\e (t) \\
\hspace{6em}+\sigma\big(Y_{\rm M}(t),\Theta(t\bmod \tau)\bar{X}_T(t;\tau_t,Y_{\rm M}(\tau_t))\big) 
-C Y_{\rm M}(t)\!-\!D \Theta(t\bmod \tau) \bar{X}_T(t;\tau_t,Y_{\rm M}(\tau_t))  \Big ]dW(t),\\
d\e(t)= \Big[\cA_\infty \e (t)+  ( \cA_{T,\tau}(t)-\cA_\infty ) \e(t)- B \Theta(t\bmod \tau)\e (t)\\
\hspace{4em}-  b\big(Y_{\rm M}(t),\Theta(t\bmod\tau)\bar{X}_T(t;\tau_t,Y_{\rm M}(\tau_t))\big) + A Y_{\rm M}(t)+B \Theta(t\bmod \tau) \bar{X}_T(t;\tau_t,Y_{\rm M}(\tau_t)) \Big ]dt\\\ns\ds 
\hspace{4em}+ \Big[ \cC_\infty \e (t)+  ( \cC_{T,\tau}(t)-\cC_\infty ) \e(t)- D \Theta(t\bmod \tau)\e (t) \\
\hspace{4em}  -\sigma\big(Y_{\rm M}(t),\Theta(t\bmod \tau)\bar{X}_T(t;\tau_t,Y_{\rm M}(\tau_t))\big) 
+C Y_{\rm M}(t)+D \Theta(t\bmod \tau) \bar{X}_T(t;\tau_t,Y_{\rm M}(\tau_t))  \Big ]dW(t),\\
Y_{\rm M}(0)=x,\q \e(\tau_t)=0.
\end{cases}
\end{align*}}
Moreover, by \cite[Chapter 11, Section 3]{Dynkin-1965}, the following stopped system admits a unique global solution: 
{\small
\begin{align*}
\begin{cases}
d y(t)=  1_{\{t<\zeta\}} \Big[\cA_\infty y (t) \! + \! ( \cA_{T,\tau}(t)-\cA_\infty ) y(t)\!+\! B \Theta(t\bmod \tau)\psi (t)\\
\hspace{6em}+  b\big(Y_{\rm M}(t),\Theta(t\bmod\tau)\bar{X}_T(t;\tau_t,Y_{\rm M}(\tau_t))\big)\! -\! A Y_{\rm M}(t)-B \Theta(t\bmod \tau) \bar{X}_T(t;\tau_t,Y_{\rm M}(\tau_t)) \Big ]dt\\\ns\ds  
\hspace{3.5em}+ 1_{\{t<\zeta\}}\Big[ \cC_\infty y(t)+  ( \cC_{T,\tau}(t)-\cC_\infty ) y(t)+ D \Theta(t\bmod \tau)\psi (t) \\
\hspace{6em}+\sigma\big(Y_{\rm M}(t),\Theta(t\bmod \tau)\bar{X}_T(t;\tau_t,Y_{\rm M}(\tau_t))\big) -C Y_{\rm M}(t)\!-\!D \Theta(t\bmod \tau) \bar{X}_T(t;\tau_t,Y_{\rm M}(\tau_t))  \Big ]dW(t),\\
d\psi(t)= 1_{\{t<\zeta\}}\Big[\cA_\infty \psi (t)+  ( \cA_{T,\tau}(t)-\cA_\infty ) \psi(t)- B \Theta(t\bmod \tau)\psi (t)\\
\hspace{6em} - b\big(Y_{\rm M}(t),\Theta(t\bmod\tau)\bar{X}_T(t;\tau_t,Y_{\rm M}(\tau_t))\big) + A Y_{\rm M}(t)+B \Theta(t\bmod \tau) \bar{X}_T(t;\tau_t,Y_{\rm M}(\tau_t)) \Big ]dt\\\ns\ds 
\hspace{3.5em}+ 1_{\{t<\zeta\}}\Big[ \cC_\infty \psi (t)+  ( \cC_{T,\tau}(t)-\cC_\infty ) \psi(t)- D \Theta(t\bmod \tau)\psi (t) \\
\hspace{6em}  -\sigma\big(Y_{\rm M}(t),\Theta(t\bmod \tau)\bar{X}_T(t;\tau_t,Y_{\rm M}(\tau_t))\big) +C Y_{\rm M}(t)+D \Theta(t\bmod \tau) \bar{X}_T(t;\tau_t,Y_{\rm M}(\tau_t))  \Big ]dW(t),\\
y(0)=x,\q \psi(\tau_t)=\e(\tau_t\wedge \zeta),\q \e(\tau_t)=0.
\end{cases}
\end{align*}}
With these preparations, we now proceed to prove Theorem \ref{Theorem 2.3}.

\begin{proof}[Proof of Theorem \ref{Theorem 2.3}]
For $\psi(\cdot)$, standard estimates for linear SDEs yield{\small 
\begin{align}
\mE |\psi(t)|^2&\!\leq\! K \mE\! \int_{\tau_t}^{t}\! 1_{\{s<\zeta\}}\Big( \big|  b\big(Y_{\rm M}(s),\Theta(s \!\bmod\!\tau)\bar{X}_T(s;\tau_s,Y_{\rm M}(\tau_s))\big) \!-\! A Y_{\rm M}(s)\!-\!B \Theta(s\!\bmod\! \tau) \bar{X}_T(s;\tau_s,Y_{\rm M}(\tau_s))\big|^2 \nonumber\\&
\qq\qq \qq   \!+ \! \big| \sigma\big(Y_{\rm M}(s),\Theta(s\!\bmod\! \tau)\bar{X}_T(s;\tau_s,Y_{\rm M}(\tau_s))\big)\!-\!C Y_{\rm M}(s)\!-\!D \Theta(s\!\bmod\! \tau) \bar{X}_T(s;\tau_s,Y_{\rm M}(\tau_s)) \big|^2\Big) ds. \label{pr-th3.1-eq3}
\end{align}}
Fixing any $0\leq t_0\leq t_1<\infty$ and applying It\^o's formula to $s\mapsto \langle P_\infty y(t),y(t)  \rangle$ over $[t_0,t_1]$,  we obtain
\begin{align}
& \mE\langle P_\infty y(t_1),y(t_1)  \rangle -  \mE\langle P_\infty y(t_0),y(t_0)  \rangle \nonumber\\=&
\mE\int_{t_0}^{t_1} 1_{\{t<\zeta\}}\Big( -\left\langle \big(Q+\cK(P_\infty)^\top P_\infty \cK(P_\infty)\big)y(t),y(t)\right \rangle \nonumber\\
&+ 2 \big \langle P_\infty \big[  ( \cA_{T,\tau}(t)-\cA_\infty ) y(t)\!+\! B \Theta(t\bmod \tau)\psi (t)\nonumber\\
&+  b\big(Y_{\rm M}(t),\Theta(t\bmod\tau)\bar{X}_T(t;\tau_t,Y_{\rm M}(\tau_t))\big)\! -\! A Y_{\rm M}(t)-B \Theta(t\bmod \tau) \bar{X}_T(t;\tau_t,Y_{\rm M}(\tau_t)) \big] , y(t) \big \rangle \nonumber\\
&+ 2 \big\langle \cC_\infty^\top  P_\infty \big[  ( \cC_{T,\tau}(t)-\cC_\infty ) y(t)+ D \Theta(t\bmod \tau)\psi (t)\nonumber \\
&+\sigma\big(Y_{\rm M}(t),\Theta(t\bmod \tau)\bar{X}_T(t;\tau_t,Y_{\rm M}(\tau_t))\big) 
-C Y_{\rm M}(t)\!-\!D \Theta(t\bmod \tau) \bar{X}_T(t;\tau_t,Y_{\rm M}(\tau_t)) \big] ,y(t) \big\rangle \nonumber\\
& +\big\langle P_\infty \big[( \cC_{T,\tau}(t)-\cC_\infty ) y(t)+ D \Theta(t\bmod \tau)\psi (t) +\sigma\big(Y_{\rm M}(t),\Theta(t\bmod \tau)\bar{X}_T(t;\tau_t,Y_{\rm M}(\tau_t))\big)\nonumber \\
& \q -C Y_{\rm M}(t)\!-\!D \Theta(t\bmod \tau) \bar{X}_T(t;\tau_t,Y_{\rm M}(\tau_t)) \big] ,	 ( \cC_{T,\tau}(t)-\cC_\infty ) y(t)+ D \Theta(t\bmod \tau)\psi (t)\nonumber \\
&\q+\sigma\big(Y_{\rm M}(t),\Theta(t\bmod \tau)\bar{X}_T(t;\tau_t,Y_{\rm M}(\tau_t))\big) 
-C Y_{\rm M}(t)\!-\!D \Theta(t\bmod \tau) \bar{X}_T(t;\tau_t,Y_{\rm M}(\tau_t))\big \rangle \Big)dt.\label{pr-th3.1-eq1}
\end{align}
Let $\Psi(t) := \mE \langle P_\infty y(t), y(t) \rangle$ for $t\geq 0$. Then
\begin{equation}\label{pr-th3.1-eq5}
\lambda_{\max}(P_\infty) \mE |y(t)|^2\geq \Psi(t)\geq \lambda_{\min}(P_\infty) \mE |y(t)|^2\qq \mbox{ for } t\geq 0. 
\end{equation}
From \eqref{pr-th3.1-eq1}, we obtain{\small 
\begin{equation}\label{pr-th3.1-eq2}
\begin{aligned}
&	\Psi(t_1)-\Psi(t_0)\\&\leq \!\int_{t_0}^{t_1}\! \Big[\!-\!2\lambda_{\min}(P_\infty) \lambda_\infty \mE |y(t)|^2 +K \big( \big| \cA_{T,\tau}(t)\!-\!\cA_\infty\big|+\big| \cC_{T,\tau}(t)\!-\!\cC_\infty\big|\!+\!\big| \cC_{T,\tau}(t)\!-\!\cC_\infty\big| ^2 \big) \mE |y(t)|^2\!+\! K \mE \big[|\psi(t)||y(t)|\big] \\ 
&\qq \q \!\!+\!K \mE |\psi(t)|^2 \! +\! K\mE \big[ 1_{\{t<\zeta\}}\!\big| b\big(Y_{\rm M}(t),\Theta(t\!\bmod\!\tau)\bar{X}_T(t;\tau_t,Y_{\rm M}(\tau_t))\big)\!-\! A Y_{\rm M}(t)\!-\!B \Theta(t\!\bmod\! \tau) \bar{X}_T(t;\tau_t,Y_{\rm M}(\tau_t))  \big||y(t)|\big] \\
&\qq \q \!\! +K\mE \big[1_{\{t<\zeta\}}\big| \sigma\big(Y_{\rm M}(t),\Theta(t\!\bmod\! \tau)\bar{X}_T(t;\tau_t,Y_{\rm M}(\tau_t))\big)\!-\!C Y_{\rm M}(t)\!-\!D \Theta(t\!\bmod\! \tau) \bar{X}_T(t;\tau_t,Y_{\rm M}(\tau_t))\big||y(t)| \big]\\
&\qq\q\! +K\mE \big[ 1_{\{t<\zeta\}}\big| \sigma\big(Y_{\rm M}(t),\Theta(t\bmod \tau)\bar{X}_T(t;\tau_t,Y_{\rm M}(\tau_t))\big)\\
&\qq\q\!\!-C Y_{\rm M}(t)\!-\!D \Theta(t\bmod \tau) \bar{X}_T(t;\tau_t,Y_{\rm M}(\tau_t))\big|^2 \big]\Big]dt.
\end{aligned}
\end{equation}}
By Cauchy's inequality with $\delta > 0$: 
{\small
\begin{align*}
\begin{cases}
	\mE \big(|\psi(t)||y(t)| \big)\leq \delta \mE |\psi(t)|^2+  \frac{1}{4\delta} \mE |y(t)|^2,\\ \ns\ds 
	\mE \big( 1_{\{t<\zeta\}}\big| b\big(Y_{\rm M}(t),\Theta(t\bmod\tau)\bar{X}_T(t;\tau_t,Y_{\rm M}(\tau_t))\big)
	\!-\! A Y_{\rm M}(t)\!-\!B \Theta(t\bmod \tau) \bar{X}_T(t;\tau_t,Y_{\rm M}(\tau_t))  \big||y(t)|\big)\\
	\leq \delta \mE \big( 1_{\{t<\zeta\}}\big| b\big(Y_{\rm M}(t),\Theta(t\bmod\tau)\bar{X}_T(t;\tau_t,Y_{\rm M}(\tau_t))\big)\\
	\qq - \!A Y_{\rm M}(t)\!-\!B \Theta(t\bmod \tau) \bar{X}_T(t;\tau_t,Y_{\rm M}(\tau_t))  \big|^2\big) \!+\!\frac{1}{4\delta}  \mE |y(t)|^2,\\ \ns\ds
	\mE \big( 1_{\{t<\zeta\}}\big| \sigma\big(Y_{\rm M}(t),\Theta(t\bmod\tau)\bar{X}_T(t;\tau_t,Y_{\rm M}(\tau_t))\big)
	\!-\! C Y_{\rm M}(t)\!-\!D \Theta(t\bmod \tau) \bar{X}_T(t;\tau_t,Y_{\rm M}(\tau_t))  \big||y(t)|\big)\\
	\leq \delta \mE \big( 1_{\{t<\zeta\}}\big| \sigma\big(Y_{\rm M}(t),\Theta(t\bmod\tau)\bar{X}_T(t;\tau_t,Y_{\rm M}(\tau_t))\big)\\
	\qq - C Y_{\rm M}(t)\!-\!D \Theta(t\bmod \tau) \bar{X}_T(t;\tau_t,Y_{\rm M}(\tau_t))  \big|^2\big) \!+\!\frac{1}{4\delta}  \mE |y(t)|^2.
\end{cases}
\end{align*}}
Choose $\delta = \frac{3K}{2 \lambda_{\min}(P_\infty) \lambda_\infty}$ and substitute into \eqref{pr-th3.1-eq2}, 
where $K$ is the fixed generic constant in \eqref{pr-th3.1-eq2}, yielding {\small 
\begin{align}
&\Psi(t_1)-\Psi(t_0)\nonumber\\&\leq \int_{t_0}^{t_1} \Big[\mE |y(t)|^2\Big(-2\lambda_{\min}(P_\infty) \lambda_\infty  
+K \big( \big| \cA_{T,\tau}(t)-\cA_\infty\big|+\big| \cC_{T,\tau}(t)-\cC_\infty\big| +\big| \cC_{T,\tau}(t)-\cC_\infty\big|^2  \big) + \frac{3}{4\delta} K   \Big) \nonumber\\
&\qq \q +(K+K\delta) \mE |\psi(t)|^2  + K \delta\mE \big( 1_{\{t<\zeta\}}\big| b\big(Y_{\rm M}(t),\Theta(t\bmod\tau)\bar{X}_T(t;\tau_t,Y_{\rm M}(\tau_t))\big)\! \nonumber\\
&\qq \q-\! A Y_{\rm M}(t)-B \Theta(t\bmod \tau) \bar{X}_T(t;\tau_t,Y_{\rm M}(\tau_t))  \big|^2\big)  +K \delta \mE \big(1_{\{t<\zeta\}}\big| \sigma\big(Y_{\rm M}(t),\Theta(t\bmod \tau)\bar{X}_T(t;\tau_t,Y_{\rm M}(\tau_t))\big) \nonumber\\
&\qq\q-C Y_{\rm M}(t)\!-\!D \Theta(t\bmod \tau) \bar{X}_T(t;\tau_t,Y_{\rm M}(\tau_t))\big|^2 \big)  +K\mE \big(1_{\{t<\zeta\}}\big| \sigma\big(Y_{\rm M}(t),\Theta(t\bmod \tau)\bar{X}_T(t;\tau_t,Y_{\rm M}(\tau_t))\big) \nonumber\\
&\qq\q-C Y_{\rm M}(t)\!-\!D \Theta(t\bmod \tau) \bar{X}_T(t;\tau_t,Y_{\rm M}(\tau_t))\big|^2 \big)\Big]dt \nonumber\\
&\leq \int_{t_0}^{t_1} \Big[ \mE |y(t)|^2 \Big(-\frac{3}{2}\lambda_{\min}(P_\infty) \lambda_\infty +K \big( \big| \cA_{T,\tau}(t)-\cA_\infty\big|+\big| \cC_{T,\tau}(t)-\cC_\infty\big| +\big| \cC_{T,\tau}(t)-\cC_\infty\big|^2 \big) \Big) \nonumber\\
&\qq \q +  \Big( K\!+ \!\frac{3K^2}{2 \lambda_{\min}(P_\infty)\lambda_\infty }\Big) \mE |\psi(t)|^2 \! +\! \frac{3K^2}{2 \lambda_{\min}(P_\infty)\lambda_\infty }\mE \big[ 1_{\{t<\zeta\}}\big| b\big(Y_{\rm M}(t),\Theta(t\bmod \tau)\bar{X}_T(t;\tau_t,Y_{\rm M}(\tau_t))\big)\nonumber\\
&\qq\q-A Y_{\rm M}(t)\!-\!B \Theta(t\bmod \tau) \bar{X}_T(t;\tau_t,Y_{\rm M}(\tau_t))\big|^2 \big] \nonumber\\
&\qq\q + \Big( K+ \frac{3K^2}{2 \lambda_{\min}(P_\infty)\lambda_\infty }\Big) \mE\big[ 1_{\{t<\zeta\}}\big| \sigma\big(Y_{\rm M}(t),\Theta(t\bmod \tau)\bar{X}_T(t;\tau_t,Y_{\rm M}(\tau_t))\big)\nonumber\\
&\qq\q-C Y_{\rm M}(t)\!-\!D \Theta(t\bmod \tau) \bar{X}_T(t;\tau_t,Y_{\rm M}(\tau_t))\big|^2 \big] \Big]dt\nonumber\\
& \leq \int_{t_0}^{t_1} \Big[ \mE |y(t)|^2 \Big(-\frac{3}{2}\lambda_{\min}(P_\infty) \lambda_\infty +K(K_1+K_1^2) e^{-2\lambda_\infty (T-\tau)} \Big)\nonumber \\
&\qq\q + K  \int_{\tau_t}^{t} \Big[ \mE \big(1_{\{s<\zeta\}}\big| b\big(Y_{\rm M}(s),\Theta(s\bmod \tau)\bar{X}_T(s;\tau_s,Y_{\rm M}(\tau_s))\big)\nonumber\\
&\qq\q-\!A Y_{\rm M}(s)\!-\!B \Theta(s\!\bmod\! \tau) \bar{X}_T(s;\tau_s,\!Y_{\rm M}(\tau_s))\big|^2 \big)  \!+\! \mE \big(1_{\{s<\zeta\}}\big| \sigma\!\big(Y_{\rm M}(s),\Theta(s\!\bmod\! \tau)\bar{X}_T(s;\tau_s,\!Y_{\rm M}(\tau_s))\big)\nonumber\\
&\qq\q-C Y_{\rm M}(s)\!-\!D \Theta(s\bmod \tau) \bar{X}_T(s;\tau_s,Y_{\rm M}(\tau_s))\big|^2 \big) \Big] ds  +K \mE \big(1_{\{t<\zeta\}}\big| b\big(Y_{\rm M}(t),\Theta(t\bmod \tau)\bar{X}_T(t;\tau_t,Y_{\rm M}(\tau_t))\big)\nonumber\\
&\qq\q-\!A Y_{\rm M}(t)\!-\!B \Theta(t\!\bmod\! \tau) \bar{X}_T\!(t;\tau_t,\!Y_{\rm M}(\tau_t))\big|^2 \big)\!+\! K \mE\big(1_{\{t<\zeta\}}\big| \sigma\big(Y_{\rm M}(t),\Theta(t\!\bmod\! \tau)\bar{X}_T(t;\tau_t,\!Y_{\rm M}(\tau_t))\big)\nonumber\\
&\qq\q-C Y_{\rm M}(t)\!-\!D \Theta(t\bmod \tau) \bar{X}_T(t;\tau_t,Y_{\rm M}(\tau_t))\big|^2 \big)\Big]dt\nonumber\\
&\leq \int_{t_0}^{t_1} \Big[ \mE |y(t)|^2 \Big(-\frac{3}{2}\lambda_{\min}(P_\infty) \lambda_\infty +K(K_1+K_1^2) e^{-2\lambda_\infty (T-\tau)} \Big) \nonumber\\
&\qq\q +K(1+\tau) \mE \big(1_{\{t<\zeta\}}\big| b\big(Y_{\rm M}(t),\Theta(t\bmod \tau)\bar{X}_T(t;\tau_t,Y_{\rm M}(\tau_t))\big) -A Y_{\rm M}(t)\!-\!B \Theta(t\bmod \tau) \bar{X}_T(t;\tau_t,Y_{\rm M}(\tau_t))\big|^2 \big)\nonumber\\
&\qq\q + K(1+\tau) \mE\big(1_{\{t<\zeta\}}\big| \sigma\big(Y_{\rm M}(t),\Theta(t\bmod \tau)\bar{X}_T(t;\tau_t,Y_{\rm M}(\tau_t))\big)\nonumber\\
&\qq\q-C Y_{\rm M}(t)\!-\!D \Theta(t\bmod \tau) \bar{X}_T(t;\tau_t,Y_{\rm M}(\tau_t))\big|^2 \big)\Big]dt,\label{pr-th3.1-eq4}
\end{align}}
where the penultimate inequality follows from \eqref{pr-th3.1-eq3}, \eqref{Theta-Thetainf},
and the last inequality follows from the fact that for any integrable function  $\eta(\cd)$:
\begin{align*}
\int_{t_0}^{t_1}\!\int_{\tau_s}^{s} \eta(r)\,dr\,ds
= \int_{t_0}^{t_1} \eta(r)\!\left(\int_{t_0}^{t_1}\mathbf{1}_{\{\tau_s \le r \le s\}}\,ds\right)dr
\le \int_{t_0}^{t_1} \eta(r)\!\left(\int_{r}^{(r+\tau)\wedge t_1} ds\right)dr
\le \tau \int_{t_0}^{t_1} \eta(r)dr.
\end{align*}
By Assumption \ref{H2} and \eqref{H2-esti}, we get that
\begin{align*}
&\big| b\big(Y_{\rm M}(t),\Theta(t\bmod \tau) \bar{X}_T(t;\tau_t,Y_{\rm M}(\tau_t))\big)-A Y_{\rm M}(t)\!-\!B \Theta(t\bmod \tau) \bar{X}_T(t;\tau_t,Y_{\rm M}(\tau_t))\big|^2\\
&+ \big| \sigma\big(Y_{\rm M}(t),\Theta(t\bmod \tau) \bar{X}_T(t;\tau_t,Y_{\rm M}(\tau_t))\big)-C Y_{\rm M}(t)\!-\!D \Theta(t\bmod \tau) \bar{X}_T(t;\tau_t,Y_{\rm M}(\tau_t))\big|^2\\
& \leq K\(\sum_{i=2}^{k} \big(|Y_{\rm M}(t)|+ \big| \bar{X}_T(t;\tau_t,Y_{\rm M}(\tau_t))\big| \big)^i \)^2\\
&\leq K \sum_{i=2}^{k}\(|Y_{\rm M}(t)|^{2i}+ \big| \bar{X}_T(t;\tau_t,Y_{\rm M}(\tau_t))\big|^{2i}  \).
\end{align*}
Thus,\vspace{-2mm}
\begin{align*}
&\mE \Big[ 1_{\{t<\zeta\}}\big| b\big(Y_{\rm M}(t),\Theta(t\bmod \tau) \bar{X}_T(t;\tau_t,Y_{\rm M}(\tau_t))\big)-A Y_{\rm M}(t)\!-\!B \Theta(t\bmod \tau) \bar{X}_T(t;\tau_t,Y_{\rm M}(\tau_t))\big|^2\\
&+1_{\{t<\zeta\}} \big| \sigma\big(Y_{\rm M}(t),\Theta(t\bmod \tau) \bar{X}_T(t;\tau_t,Y_{\rm M}(\tau_t))\big)-C Y_{\rm M}(t)\!-\!D \Theta(t\bmod \tau) \bar{X}_T(t;\tau_t,Y_{\rm M}(\tau_t))\big|^2\Big]\\
&\leq K \sum_{i=2}^{k} \(\mE \big[ 1_{\{t<\zeta\}} |Y_{\rm M}(t)|^{2i}\big]+ \mE \Big[ 1_{\{t<\zeta\}}  \big| \bar{X}_T(t;\tau_t,Y_{\rm M}(\tau_t))\big|^{2i}\Big]  \)\\
&\leq K \sum_{i=2}^{k} \(\mE \big[ 1_{\{t<\zeta\}} |y(t)|^{2i}\big]+ \mE \Big[ 1_{\{t<\zeta\}}  \big| \bar{X}_T(t\wedge \zeta;\tau_t,Y_{\rm M}(\tau_t\wedge \zeta))\big|^{2i}\Big]  \)\\
&\leq K \sum_{i=2}^{k} \(\mE  |y(t)|^{2i}+ \mE  \big|y(\tau_t )\big|^{2i}  \)\\
&\leq K \sum_{i=1}^{k-1} r^{2i} \times \(\mE  |y(t)|^2+ \mE  \big|y(\tau_t )\big|^2  \),
\end{align*}
where the penultimate inequality follows from standard estimates for linear SDEs, and the last inequality uses the definition of the stopping time \eqref{stopping}, $\tau_t \leq t$, and $y(t) = Y_{\rm M}(t \wedge \zeta)$. Combining with \eqref{pr-th3.1-eq5} and \eqref{pr-th3.1-eq4}  implies 
\begin{align*}
&\Psi(t_1)-\Psi(t_0)\\&\leq  \int_{t_0}^{t_1} \Big[\mE |y(t)|^2 \Big(-\frac{3}{2}\lambda_{\min}(P_\infty) \lambda_\infty +K(K_1+K_1^2) e^{-2\lambda_\infty (T-\tau)} \Big)  + K(1+\tau)  \sum_{i=1}^{k-1} r^{2i}  \big(\mE  |y(t)|^2+ \mE  \big|y(\tau_t )\big|^2  \big)\Big]\\
&\leq   \int_{t_0}^{t_1} \Big[ \Psi(t) \Big( -\frac{3\lambda_{\min}(P_\infty) \lambda_\infty}{2 \lambda_{\max}(P_\infty)}  +K(K_1+K_1^2) e^{-2\lambda_\infty (T-\tau)} +K(1+\tau) \sum_{i=1}^{k-1} r^{2i} \Big)  + K(1+\tau)\sum_{i=1}^{k-1} r^{2i}\,   \Psi(\tau_t )  \Big]dt.
\end{align*}
Recall $\lambda^* := \frac{\lambda_{\min}(P_\infty) \lambda_\infty}{\lambda_{\max}(P_\infty)}$ and choose $$T-\tau \geq  \frac{1}{2\lambda_\infty}\ln\!\left(\frac{2K(K_1+K_1^2)\lambda_{\max}(P_\infty)}{\lambda_{\min}(P_\infty)\,\lambda_\infty}\right).$$
Then\vspace{-2mm}
\begin{align*}
\Psi(t_1)-\Psi(t_0)&\leq \int_{t_0}^{t_1} \[ -\(\lambda^*-K(1+\tau) \sum_{i=1}^{k-1} r^{2i} \)\Psi(t) +K(1+\tau) \sum_{i=1}^{k-1} r^{2i}\, \Psi(\tau_t)   \]dt.\vspace{-2mm}
\end{align*}
Since $t_0$ and $t_1$ are arbitrary, we  get \vspace{-2mm}
\begin{equation*}
\Psi(t)'\leq  -\(\lambda^*-K(1+\tau)\sum_{i=1}^{k-1} r^{2i} \)\Psi(t) +K(1+\tau)\sum_{i=1}^{k-1} r^{2i} \Psi(\tau_t),\q t\geq 0.
\end{equation*}
By Lemma \ref{lm-tau}, with $r < \min\left\{ \frac{1}{\tau}, \frac{\lambda^*}{4},1, -{K\over 2}(1+\tau)+{1\over 2}\sqrt{K^2(1+\tau)^2+4} \right\}$, we obtain $K(1+\tau)\sum_{i=1}^{k-1}r^{2i}\leq r$ and 
\begin{equation*}
\Psi(t)\leq 2 \Psi(0) e^{-\frac{\lambda^*}{2}t},\q t\geq 0.
\end{equation*}
Therefore, we have\vspace{-2mm}
\begin{equation}\label{pr-th3.1-eq6}
\begin{aligned}
\mE |Y_{\rm M}(t\wedge \zeta)|^2\leq \frac{1}{\lambda_{\min}(P_\infty)}\Psi(t)\leq \frac{2 \langle P_\infty x,x \rangle}{\lambda_{\min}(P_\infty)} e^{-\frac{\lambda^*}{2}t},\q t\geq 0.
\end{aligned}\vspace{-2mm}
\end{equation}

Finally, we show that $\zeta = \infty$ almost surely. Observe that  \vspace{-2mm}
\begin{align*}
\mE |Y_{\rm M}( \zeta)|^2= \mE \lim_{t\to\infty}  |Y_{\rm M}(t\wedge \zeta)|^2\leq \liminf_{t\to \infty} \mE  |Y_{\rm M}(t\wedge \zeta)|^2\leq   \liminf_{t\to \infty} \frac{2 \langle P_\infty x,x \rangle}{\lambda_{\min}(P_\infty)} e^{-\frac{\lambda^*}{2}t}=0,\vspace{-2mm}
\end{align*}
where the second inequality follows from Fatou's lemma. Hence, $Y_{\rm M}(\zeta) = 0$ almost surely. 

If $\dbP(\zeta < \infty) > 0$, then $\mE |Y_{\rm M}(\zeta)|^2> 0$, a contradiction. Therefore, $\dbP(\zeta = \infty) = 1$. Combining with \eqref{pr-th3.1-eq6}, we arrive at \vspace{-2mm}
\begin{align*}
\mE |Y_{\rm M}(t )|^2\leq  \frac{2 \langle P_\infty x,x \rangle}{\lambda_{\min}(P_\infty)} e^{-\frac{\lambda^*}{2}t},\q t\geq 0.\vspace{-2mm}
\end{align*}
This completes the proof.
\end{proof}

\subsection{Proof of the preliminaries}

\subsubsection{Proof of Lemma \ref{lm5.1}}\label{app-sec1}

\begin{proof}[Proof of Lemma \ref{lm5.1}]
Fix any $T>0$, and consider Problem (SLQ)$_T$ with state equation 
\begin{align*}
\begin{cases}
dX(t)=\big( AX(t)+Bu(t)\big)dt+ \big(CX(t)+Du(t)\big)dW(t),\q t\in [0,T],\\
X(0)=x,
\end{cases}
\end{align*}
and the cost functional 
\begin{align*}
J_{T}(x ; u(\cdot))=\frac{1}{2} \mathbb{E} \[\int_0^{T}[\langle Q X(t), X(t)\rangle+\langle R u(t), u(t)\rangle]d t+ \langle GX(T),X(T) \rangle\].
\end{align*}
Note that \vspace{-2mm}
\begin{align}\label{pr-lm5.1-eq1}
\cU_{ad}(x)=\{u \,|\, u\in L_{\dbF}^2(0,T;\dbR^m) \}= \{\cK(P_\infty)X+ v\,|\,  v\in L_{\dbF}^2(0,T;\dbR^m)\},\vspace{-2mm}
\end{align}
and\vspace{-2mm}
\begin{align*}
\frac{1}{2}\langle \Sigma(T)x,x \rangle=\inf_{u\in \cU_{ad}(x)} J_{T}(x ; u(\cdot)).\vspace{-2mm}
\end{align*}
Define\vspace{-2mm}
$$\wt A= A+B\cK(P_\infty),\q \wt C= C+D\cK(P_\infty),\q \wt Q= Q+ \cK(P_\infty)^\top R \cK(P_\infty),\q \wt S=R\cK(P_\infty),\vspace{-2mm}$$
and consider  Problem $\wt{({\rm SLQ})}_T$ with the state equation 
\begin{align*}
\begin{cases}
d\wt X(t)=\big( \wt A \wt X(t)+Bv(t)\big)dt+ \big(\wt C\wt X(t)+Dv(t)\big)dW(t),\q t\in [0,T],\\
\wt X(0)=x,
\end{cases}
\end{align*}
and the cost functional 
\begin{align*}
\wt J_{T}(x;v)=\mE \left[ \int_{0}^{T} \left \langle    \begin{pmatrix}
\wt Q & \wt S^\top \\ \wt S & R
\end{pmatrix} \begin{pmatrix}
\wt X \\v
\end{pmatrix}, 
\begin{pmatrix}
\wt X \\v
\end{pmatrix}  \right\rangle dt + \big \langle G \wt X(T),\wt X(T) \big\rangle  \right].
\end{align*}
Then we have\vspace{-4mm}
\begin{align*}
&\wt X(t;x,v)=X(t;\cK(P_\infty)X+v),\q \forall t\in [0,T],\\
& J_T(x,\cK(P_\infty)X+v)=\wt J_T(x,v), \q \forall (x,v)\in \dbR^n\times L_{\dbF}^2(0,T;\dbR^m).
\end{align*}
Hence, from \eqref{pr-lm5.1-eq1}, we  get\vspace{-2mm}
\begin{align*}
\frac{1}{2}\langle \Sigma(T)x,x \rangle=\inf_{u\in \cU_{ad}(x)} J_{T}(x ; u(\cdot))=\inf_{v\in L_{\dbF}^2(0,T;\dbR^m)}J_T(x,\cK(P_\infty)X+v)=
\inf_{v\in L_{\dbF}^2(0,T;\dbR^m)}\wt J_T(x,v).\vspace{-2mm}
\end{align*}
Since $Q \in \dbS^n_{>0}$ and $R \in \dbS^m_{>0}$, we also have \vspace{-2mm}
\begin{align*}
\inf_{v\in L_{\dbF}^2(0,T;\dbR^m)}\wt J_T(x,v)=\frac{1}{2}\langle \wt \Sigma(T)x,x \rangle,\vspace{-2mm}
\end{align*}
where $\wt \Sigma(\cdot)$ satisfies the differential Riccati equation:
\begin{equation}\label{wt-DRE}
\begin{cases}
\dot{\wt \Sigma}(t)- \wt\Sigma(t) \wt A -\wt A^\top \wt\Sigma(t)- \wt C^\top \wt\Sigma(t) \wt C- \wt Q \\
+ \big(\wt \Sigma(t)B+ \wt C^\top \Sigma(t) D+\wt S^\top \big) \big(R+D^\top \wt\Sigma(t) D\big)^{-1} \big(B^\top \wt \Sigma (t) + D^\top \wt \Sigma (t) \wt C+ \wt S\big)=0, \q t\geq 0,\\
\wt \Sigma(0)=G,
\end{cases}
\end{equation}
Consequently, $\wt \Sigma(T) = \Sigma(T)$. Since $T \geq 0$ is arbitrary, it follows that $\wt \Sigma(t; G) = \Sigma(t; G)$ for all $t \geq 0$.

We now establish the desired estimate. Note that $\wt A = \cA_\infty$ and $\wt C = \cC_\infty$.
Apply It\^o's formula to $s \mapsto \Phi(s)^\top \wt \Sigma(t - s) \Phi(s)$ for $0 \leq s \leq t < \infty$, where $\Phi(\cdot)$ is defined in \eqref{lm2.2-Phi}:{\small  
\begin{align*}
\wt \Sigma(t)\!=\!\mE \[ \Phi(t)^\top\! G \Phi(t) \] \!
+\! \mE\! \int_{0}^{t}\!\! \Phi(t\!-\!s)^\top\! \big[\wt Q\!-\! \big(\wt \Sigma(s)B\!+\! \wt C^\top\! \Sigma(s) D\!+\!\wt S^\top\!\big) \big(R\!+\!\!D^\top\! \wt\Sigma(s) D\big)^{-1}\! \big(B^\top\! \wt \Sigma (s)\! +\! D^\top\! \wt \Sigma (s) \wt C\!+\! \wt S\big)\big]\Phi(t\!-\!s)ds.\nonumber
\end{align*}}
Therefore,\vspace{-2mm}
\begin{align*}
|\wt \Sigma(t)|&\leq \mE \[ \big|\Phi(t)^\top G \Phi(t)\big| \] +\mE \int_{0}^{t}  \big|\Phi(t-s)^\top \wt Q  \Phi(t-s)\big|ds\\
&\leq |G|K_0e^{-2\lambda_\infty t}+ |\wt Q| \int_{0}^{t} K_0e^{-2\lambda_\infty (t-s)}ds\\
&\leq K_0\(|G|+\frac{|\wt Q|}{2\lambda_\infty} \).\vspace{-2mm}
\end{align*}
This completes the proof.
\end{proof}
\subsubsection{Proof of Lemma \ref{lm2.1}}\label{app-sec2}

Before proving Lemma \ref{lm2.1},  we define the following set:\vspace{-2mm}
\begin{align*}
\cD\= \{ G\in \dbS^n_{\geq 0}\,|\, \lim_{t\to \infty} \Sigma(t;G)=P_\infty \}.\vspace{-2mm}
\end{align*}
We now proceed with the proof.

\begin{proof}[Proof of Lemma \ref{lm2.1}]
We divide the proof into two steps.

{\bf Step 1.} We show that estimate \eqref{Con-Riccati} holds for any $G \in \cD$.

Let $\Pi(\cd)\=P_\infty-\Sigma(\cd)$. Then\vspace{-2mm}
\begin{align*}
\dot{\Pi} &= \Pi A+ A^\top \Pi+ C^\top \Pi C - \cS(P_\infty)^\top \cR(P_\infty)^{-1} \cS(P)+ \cS(\Sigma)^\top \cR (\Sigma)^{-1} \cS(\Sigma)\\
&=  \Pi [A+B\cK(P_\infty)]+ [A+B\cK(P_\infty)]^\top \Pi+ [C+D\cK(P_\infty)]^\top \Pi [C+D\cK(P_\infty)] \\
&\q + \cS(\Sigma)^\top \cR (\Sigma)^{-1} \cS(\Sigma) - \cK(P_\infty)^\top \cS (\Pi)- \cS(\Pi)^\top \cK(P_\infty)-  \cK(P_\infty) D^\top \Pi D \cK(P_\infty).\vspace{-2mm}
\end{align*}
Define\vspace{-2mm}
\begin{align*}
& f(t,\Pi):=   \cS(\Sigma)^\top \cR (\Sigma)^{-1} \cS(\Sigma)  - \cK(P_\infty)^\top \cS (\Pi)- \cS(\Pi)^\top \cK(P_\infty)-  \cK(P_\infty) D^\top \Pi D \cK(P_\infty).
\end{align*}
Then $\Pi$ satisfies\vspace{-2mm}
\begin{align}\label{Pi}
\begin{cases}
\dot{\Pi}(t)= \Pi(t)\cA_\infty  + \cA_\infty^\top \Pi(t) +\cC_\infty^\top \Pi (t) \cC_\infty +f(t,\Pi(t)),\q t\geq 0,\\
\Pi(0)=P_\infty-G.
\end{cases}\vspace{-2mm}
\end{align}
A direct computation shows that $f(t, \Pi)$ is quadratic in $\Pi$:
\begin{align*}
f(t,\Pi)&= - \big(  \cS(\Pi)^\top + \cK (P_\infty) D^\top \Pi D  \big) \cK(P_\infty) + \big( \cK(P_\infty)^\top + \cS(\Sigma) \cR(\Sigma)^{-1} \big)\cS(\Sigma)\\
&=   \big(  \cS(\Pi)^\top + \cK (P_\infty) D^\top \Pi D  \big) \cR(P_\infty)^{-1}\big( \cS(\Pi)+ \cS(\Sigma) \big)   + \big( \cK(P_\infty)^\top + \cS(\Sigma)^\top \cR(\Sigma)^{-1} \big)\cS(\Sigma)\\
&=\big[ \big(  \cS(\Pi)^\top + \cK (P_\infty) D^\top \Pi D  \big) \cR(P_\infty)^{-1} \cS(\Pi) \big]\\
&\q + \big(  \cS(\Pi)^\top + \cK (P_\infty) D^\top \Pi D +  \cK(P_\infty)^\top    \cR(P_\infty)+ \cS(\Sigma)^\top \cR(\Sigma)^{-1}   \cR(P_\infty) \big)  \cR(P_\infty)^{-1} \cS(\Sigma)\\
&=\big[ \big(  \cS(\Pi)^\top + \cK (P_\infty) D^\top \Pi D  \big) \cR(P_\infty)^{-1} \cS(\Pi) \big]\\
&\q + \big(  -\cS(\Sigma)^\top + \cK (P_\infty) D^\top \Pi D + \cS(\Sigma)^\top \cR(\Sigma)^{-1}   \cR(P_\infty) \big)  \cR(P_\infty)^{-1} \cS(\Sigma)\\
&=\big[ \big(  \cS(\Pi)^\top + \cK (P_\infty) D^\top \Pi D  \big) \cR(P_\infty)^{-1} \cS(\Pi) \big ]\\
&\q + \big(  \cS(\Sigma)^\top  \cR(\Sigma)^{-1} D^\top \Pi D -  \cS(P_\infty)^\top \cR(P_\infty)^{-1}  D^\top \Pi D  \big)  \cR(P_\infty)^{-1} \cS(\Sigma)\\
&=\big[ \big(  \cS(\Pi)^\top + \cK (P_\infty) D^\top \Pi D  \big) \cR(P_\infty)^{-1} \cS(\Pi) \big ]\\
&\q + \big[  \big(\cS(\Sigma) -\cS(P_\infty) \big)^\top  \cR(\Sigma)^{-1}  +  \cS(P_\infty)^\top \big(\cR(\Sigma)^{-1}-\cR(P_\infty)^{-1} \big)  \big] D^\top \Pi D \cR(P_\infty)^{-1} \cS(\Sigma)\\
&=\big[ \big(  \cS(\Pi)^\top + \cK (P_\infty) D^\top \Pi D  \big) \cR(P_\infty)^{-1} \cS(\Pi) \big]\\
&\q + \big[  -\cS(\Pi)^\top  \cR(\Sigma)^{-1}  +  \cS(P_\infty)^\top  \cR(\Sigma)^{-1} \big(\cR(P_\infty) -\cR(\Sigma) \big)\cR(P_\infty)^{-1}  \big] D^\top \Pi D \cR(P_\infty)^{-1} \cS(\Sigma)\\
&=\big [ \big(  \cS(\Pi)^\top + \cK (P_\infty) D^\top \Pi D  \big) \cR(P_\infty)^{-1} \cS(\Pi) \big ]\\
&\q + \big[  -\cS(\Pi)^\top  \cR(\Sigma)^{-1}  +  \cS(P_\infty)^\top  \cR(\Sigma)^{-1} D^\top \Pi D \cR(P_\infty)^{-1}  \big] D^\top \Pi D \cR(P_\infty)^{-1} \cS(\Sigma).
\end{align*}
Define $f : [0, \infty) \times \dbR^{n \times n} \to \dbR^{n \times n}$ by 
\begin{align*}
f(t,M)&:=\big [ \big(  \cS(M)^\top + \cK (P_\infty) D^\top M D  \big) \cR(P_\infty)^{-1} \cS(M) \big ]\\
&\q + \big[  -\cS(M)^\top  \cR(\Sigma(t))^{-1}  +  \cS(P_\infty)^\top  \cR(\Sigma(t))^{-1} D^\top M D \cR(P_\infty)^{-1}  \big] D^\top M D \cR(P_\infty)^{-1} \cS(\Sigma(t)).
\end{align*}
Then $f(t, 0) = 0$, and by Lemma \ref{lm5.1}, \vspace{-2mm}
\begin{align*}
|f(t,M)-f(t,N)|&\leq  |M-N|(|M|+|N|)\Big[ \big(|B|+|D||C|+|\cK(P_\infty)||D|^2\big) \frac{|B|+|D||C|}{\lambda_{\min}(R)} \Big]\\
&\q+  |M-N|(|M|+|N|) \Big[ \frac{(|B|+ |D||C|)^2 |D|^2}{\lambda_{\min}(R)^2} |\Sigma(t)| \Big]\\
&\q +   |M-N|(|M|+|N|) \Big[ \frac{ |\cS(P_\infty)|(|B|+ |D||C|) |D|^4}{\lambda_{\min}(R)^3} |\Sigma(t)| \Big]\\
&\leq \rho |M-N|(|M|+|N|),\vspace{-2mm}
\end{align*}
where \vspace{-2mm}
\begin{equation*}
\begin{aligned}
\rho:&= (|B|+|D||C|+|\cK(P_\infty)||D|^2) (|B|+|D||C|){\lambda_{\min}(R)}^{-1} + \big[(|B|+ |D||C|)^2 |D|^2{\lambda_{\min}(R)}^{-2} \\
&\q+   |\cS(P_\infty)|(|B|+ |D||C|) |D|^4 {\lambda_{\min}(R)}^{-3} \big] K_0\big(|G|+   (2\lambda_\infty)^{-1}\big|Q+ \cK(P_\infty)^\top R \cK(P_\infty)\big|\big)
\end{aligned}\vspace{-2mm}
\end{equation*}
is a constant depending only on system parameters.

Since $G \in \cD$, we have $\lim\limits_{t \to \infty} \Pi(t) = 0$. Thus, there exists $t_0 > 0$ (depending on $G$) such that\vspace{-2mm}
\begin{align*}
|\Pi(t_0)|\leq \frac{1}{2K_0}\min\Big\{\frac{1}{4K_0\rho},\frac{\lambda_\infty}{K_0\rho}\Big\},
\end{align*}
where $K_0$ and $\lambda_\infty$ are given in \eqref{nota}. Define the space \vspace{-2mm}
\begin{align*}
\cX\= \big\{ M(\cd)\in C([t_0,\infty);\dbR^{n\times n})\,\big|\, |M(t)|\leq \delta e^{-2\lambda_\infty (t-t_0)},\, \forall t \geq t_0 \big\},
\end{align*}
where $\delta > 0$ will be determined later. For any $M(\cdot) \in \cX$ and initial state $\Pi_0 \in \dbR^{n \times n}$, the linearized ODE \vspace{-2mm}
\begin{align*}
\begin{cases}
\dot{\Pi}(t)=  \Pi(t)\cA_\infty + \cA_\infty^\top \Pi(t) +\cC_\infty^\top \Pi (t) \cC_\infty +f(t,M(t)),\q t\geq t_0,\\
\Pi(t_0)=\Pi_0,
\end{cases}
\end{align*}
admits a unique solution $\Pi(\cdot) = \cT[M(\cdot)]$. 

Applying It\^o's formula to $s\to \Phi(s)^\top \Pi(t-s) \Phi(s) $ for $0\leq s\leq t-t_0 <\infty$,  where $\Phi(\cdot)$ is defined in \eqref{lm2.2-Phi}, we obtain \vspace{-2mm}
\begin{align*}
\Pi(t)=\mE \[ \Phi(t-t_0)^\top \Pi_0 \Phi(t-t_0) \]+ \mE \int_{t_0}^{t} \Phi(t-s)^\top f(s,M(s))\Phi(t-s)ds.
\end{align*}
By Lemma \ref{lm2.2},\vspace{-4mm}
\begin{align*}
|\Pi(t)|&\leq |\Pi_0|\, K_0 e^{-2\lambda_\infty (t-t_0)}+ \int_{t_0}^{t} K_0 e^{-2\lambda_\infty (t-s)}\,|f(s,M(s))|ds\\
&\leq |\Pi_0| \,K_0 e^{-2\lambda_\infty (t-t_0)}+ \int_{t_0}^{t} \rho K_0 e^{-2\lambda_\infty (t-s)} \delta^2 e^{-4\lambda_\infty (s-t_0)} ds\\
&\leq |\Pi_0| \,K_0 e^{-2\lambda_\infty (t-t_0)}+  \rho  \delta^2 K_0 e^{-2\lambda_\infty (t-t_0)} \int_{t_0}^{t} e^{-2\lambda_\infty (s-t_0)}  ds\\
&\leq K_0\(  |\Pi_0|+ \frac{  \rho \delta^2 }{2\lambda_\infty} \)e^{-2\lambda_\infty (t-t_0)}. 
\end{align*}
Taking $|\Pi_0| \leq \frac{\delta}{2K_0}$ and $\delta \leq \frac{\lambda_\infty}{K_0 \rho}$, we have 
\begin{align*}
K_0\(  |\Pi_0|+ \frac{  \rho \delta^2 }{2\lambda_\infty} \)\leq \frac{\delta}{2}+ \frac{K_0\rho \delta^2}{2\lambda_\infty}=\frac{1}{2}\delta \(1+\frac{K_0\rho }{\lambda_\infty}\delta\)\leq \delta.
\end{align*}
Thus, with $\delta \leq \frac{\lambda_\infty}{K_0 \rho}$ and $|\Pi_0| \leq \frac{\delta}{2K_0}$, $\cX$ is invariant under $\cT$. Moreover, for any $M(\cdot), N(\cdot) \in \cX$, we have
\begin{align*}
|\cT[M(\cd)]-\cT[N(\cd)]|&= \left |\mE \int_{t_0}^{t} \Phi(t-s)^\top \(f(s,M(s))-f(s,N(s))\)\Phi(t-s)ds \right |\\
&\leq \int_{t_0}^{t} K_0 e^{-2\lambda_\infty (t-s)} \big|f(s,M(s))-f(s,N(s))\big|ds\\
&\leq  K_0 \rho\int_{t_0}^{t}  e^{-2\lambda_\infty (t-s)} \big|M(s)-N(s)\big|\big(|M(s)+|N(s)|\big)ds\\
&\leq  2K_0\delta \rho\int_{t_0}^{t}  e^{-2\lambda_\infty [(t-t_0)-(s-t_0)]}   e^{-2\lambda_\infty (s-t_0)} \big|M(s)-N(s)\big|ds\\
&\leq  2K_0\delta \rho  e^{-2\lambda_\infty (t-t_0)} \int_{t_0}^{t}   \big|M(s)-N(s)\big|ds\\
&\leq  2K_0\delta \rho e^{-2\lambda_\infty (t-t_0)} \sup_{s\geq t_0}|M(s)-N(s)|.
\end{align*}
Taking $\delta = \min\left\{ \frac{1}{4K_0 \rho}, \frac{\lambda_\infty}{K_0 \rho} \right\}$, $\cT$ is a contraction on $\cX$, and the equation \vspace{-2mm}
\begin{align}\label{pr-lm2-eq2}
\begin{cases}\ds
\dot{\Pi}(t)= \Pi(t)\cA_\infty + \cA_\infty^\top \Pi(t) +\cC_\infty^\top \Pi (t) \cC_\infty +f(t,\Pi(t)),\q t\geq t_0,\\
\Pi(t_0)=\Pi_0 
\end{cases}
\end{align}
admits a unique solution in $\cX$ for $|\Pi_0| \leq \frac{\delta}{2K_0} = \frac{1}{2K_0} \min\left\{ \frac{1}{4K_0 \rho}, \frac{\lambda_\infty}{K_0 \rho} \right\}$.

Since \eqref{Pi} coincides with \eqref{pr-lm2-eq2} on $[t_0, \infty)$ and $|\Pi(t_0)| \leq \frac{1}{2K_0} \min\left\{ \frac{1}{4K_0 \rho}, \frac{\lambda_\infty}{K_0 \rho} \right\}$, we conclude \vspace{-2mm}
\begin{align*}
|\Sigma(t)-P_\infty|=|\Pi(t)|\leq \delta e^{-2\lambda_\infty (t-t_0)}=\min\Big\{\frac{1}{4K_0\rho},\frac{\lambda_\infty}{K_0\rho}\Big\}e^{-2\lambda_\infty (t-t_0)},\q \forall t\geq t_0.
\end{align*}
On the other hand, by Lemma \ref{lm5.1},
\begin{align*}
\sup_{t\in[0,t_0]}|\Pi(t)|&\leq |P_\infty|+ \sup_{t\in[0,t_0]}|\Sigma(t)| \leq |P_\infty|+ \sup_{t\in[0,\infty)}|\Sigma(t)|\\ 
&\leq  |P_\infty|+  K_0\big(|G|+     (2\lambda_\infty)^{-1}\big|Q+ \cK(P_\infty)^\top R \cK(P_\infty)\big|  \big).
\end{align*}
Let\vspace{-2mm}
\begin{equation}\label{K1}
K_1:= \max\Big \{ \big[|P_\infty|+    K_0\big(|G|+      (2\lambda_\infty)^{-1}\big|Q+ \cK(P_\infty)^\top R \cK(P_\infty)\big| \big) \big] e^{4\lambda_\infty t_0} , 
\min\Big\{\frac{1}{4K_0\rho},\frac{\lambda_\infty}{K_0\rho}\Big\} e^{2\lambda_\infty t_0}  \Big \},\vspace{-2mm}
\end{equation}
then\vspace{-2mm}
\begin{align*}
|\Sigma(t)-P_\infty|\leq K_1e^{-2\lambda_\infty t},\q \forall t\geq 0.\vspace{-2mm}
\end{align*}
This completes Step 1.

\ms

{\bf Step 2.} We prove that $\dbS^n_{\geq 0} \subset \cD$.

Fix any $G \in \dbS^n_{\geq 0}$ and $T > 0$. Consider Problem (SLQ)$_T$ with state equation 
\begin{align*}
\begin{cases}
dX(t)= [ AX(t)+Bu(t)]dt+ [CX(t)+Du(t)]dW(t),\q t\in [0,T],\\
X(0)=x,
\end{cases}
\end{align*}
and the cost functional 
\begin{align*}
J_{T}(x ; u(\cdot))=\frac{1}{2} \mathbb{E} \[\int_0^{T}[\langle Q X(t), X(t)\rangle+\langle R u(t), u(t)\rangle]d t+ \langle G X(T),X(T) \rangle\].
\end{align*}
In the case, it is clear that 
\begin{align*}
\cU_{ad}(x)=\{u\,|\, u\in L_{\dbF}^2(0,T;\dbR^m) \}= \{\cK(P_\infty)X+ v \,|\,  v\in L_{\dbF}^2(0,T;\dbR^m)\},
\end{align*}
and  that\vspace{-2mm}
\begin{align*}
\frac{1}{2}\langle \Sigma(T)x,x \rangle=\inf_{u\in \cU_{ad}(x)} J_{T}(x ; u(\cdot)).\vspace{-2mm}
\end{align*}
Choosing $v = 0$, we obtain \vspace{-2mm}
\begin{align}\label{pr-lm2.1-eq1}
\frac{1}{2}\langle \Sigma(T)x,x \rangle\leq  J_{T}(x ; \cK(P_\infty)X)= \frac{1}{2}\langle \bar{\Sigma}(T)x,x \rangle,\vspace{-2mm}
\end{align}
where $\bar{\Sigma}(\cdot)$ satisfies the Lyapunov equation 
\begin{align*}
\begin{cases}
\dot{\bar{\Sigma}}(t)= Q+\bar{\Sigma} (A+B\cK(P_\infty))+  (A+B\cK(P_\infty))^\top \bar{\Sigma} \\ \hspace*{3em}+  (C+D\cK(P_\infty))^\top \bar{\Sigma}   (C+D\cK(P_\infty))+\cK(P_\infty)^\top R \cK(P_\infty), \q t\geq 0,\\
\bar{\Sigma}(0)=G.
\end{cases}
\end{align*}
Let $\Delta(\cd)=\bar{\Sigma}(\cd)-P_\infty$. From \eqref{ARE-form2}, we  get that
\begin{align*}
\begin{cases}
\dot{\Delta}(t)= \Delta(t) \cA_\infty+ \cA_\infty ^\top \Delta(t) + \cC_\infty^\top \Delta(t) \cC_\infty,\q t\geq 0,\\
\Delta(0)=G-P_\infty.
\end{cases}
\end{align*}
Applying It\^o's formula to $s\to \Phi(s)^\top \Delta(t-s) \Phi(s) $ for $0\leq s\leq t <\infty$ (with $\Phi(\cdot)$ as in \eqref{lm2.2-Phi}), 
we obtain that\vspace{-2mm}
\begin{align}\label{pr-lm2.1-eq2}
\Delta(t)=\mE \Phi(t)^\top (G-P_\infty)\Phi(t).\vspace{-2mm}
\end{align}
From this, Lemma \ref{lm3.2} and \eqref{pr-lm2.1-eq1},  we have\vspace{-2mm}
\begin{align}\label{pr-lm2.1-eq3}
\Sigma(t;G)\leq  \bar{\Sigma}(t;G) \text{ and } \lim_{t\to \infty} \bar{\Sigma}(t;G)=P_\infty.\vspace{-2mm}
\end{align}
On the other hand, it follows from the definition of Problem (SLQ)$_{T}$ that 
\begin{align}\label{pr-lm2.1-eq4}
\Sigma(T;0)\leq \Sigma(T;G),\q \forall T\geq 0.
\end{align}
Moreover, by Lemma \ref{MPC-lm1}, we see that\vspace{-4mm}
\begin{align}\label{pr-lm2.1-eq5}
\lim_{t\to \infty}\Sigma(t;0)=P_{\infty}.\vspace{-2mm}
\end{align}

Combine \eqref{pr-lm2.1-eq3}--\eqref{pr-lm2.1-eq5}, we obtain \vspace{-2mm}
\begin{align}\label{pr-lm2.1-ste2-eq1}
\Sigma(t;0)\leq \Sigma(t;G)\leq \bar{\Sigma}(t;G),\vspace{-2mm}
\end{align}
leading to $\lim\limits_{t\to \infty} \Sigma(t;G)= P_\infty$. This completes the proof.
\end{proof}
\vspace{-3mm}
\subsubsection{Proof of Lemma \ref{lm2.2}} \label{app-sec3}

\begin{proof}[Proof of Lemma \ref{lm2.2}]
From Lemma \ref{lm3.2} and equation \eqref{pr-lm2.1-eq2} in the proof of Lemma \ref{lm2.1}, together with the assumption $G \geq P_\infty$, we have \vspace{-2mm}
$$
\bar{\Sigma}(t;G)-P_\infty=\Delta(t)= \mE \Phi(t)^\top (G-P_\infty)\Phi(t)\geq 0,\vspace{-2mm}
$$
and \vspace{-2mm}
$$
\big|\bar{\Sigma}(t;G)-P_\infty \big|\leq |G-P_\infty|\frac{n|P_\infty|}{\lambda_{\rm min}(P_\infty)}e^{-2\lambda_\infty t}.\vspace{-1mm}
$$
On the other hand, since $G \geq P_\infty$, the monotonicity property of Problem (SLQ)$_T$ implies \vspace{-2mm}
$$
\Sigma(T;G)\geq \Sigma(T;P_\infty)=P_\infty, \q \forall T\geq 0.\vspace{-2mm}
$$
Combining this with \eqref{pr-lm2.1-eq3}, we obtain \vspace{-2mm}
$$
\big|\bar{\Sigma}(t;G)-P_\infty \big|\geq  |\Sigma(t;G)-P_\infty|.\vspace{-2mm}
$$
The desired estimate follows immediately.
\end{proof}

\end{document}